\documentclass[12pt]{article}

\usepackage{mathtools}
\usepackage{amsthm}
\usepackage{csquotes}
\usepackage{amssymb}
\usepackage{graphicx}
\usepackage{subfig}
\usepackage{soul}  
\usepackage{xcolor}
\usepackage{empheq}
\usepackage{lipsum}
\usepackage{fullpage}
\usepackage{float}

\usepackage[top= 2cm, bottom = 2 cm, left = 2.2 cm, right= 2.2 cm]{geometry}  
\usepackage[obeyspaces,hyphens,spaces]{url}
\usepackage[unicode]{hyperref}
\hypersetup{
    colorlinks=true,
    linkcolor=blue,
    citecolor=blue,
    filecolor=magenta,
    urlcolor=cyan
}
\usepackage{listings}

\definecolor{codegreen}{rgb}{0,0.6,0}
\definecolor{codegray}{rgb}{0.5,0.5,0.5}
\definecolor{codepurple}{rgb}{0.58,0,0.82}
\definecolor{backcolour}{rgb}{0.95,0.95,0.92}

\lstdefinestyle{mystyle}{
	backgroundcolor=\color{backcolour},   
	commentstyle=\color{codegreen},
	keywordstyle=\color{magenta},
	numberstyle=\footnotesize\color{codegray},
	stringstyle=\color{codepurple},
	basicstyle=\ttfamily\small,
	breakatwhitespace=false,         
	breaklines=true,                 
	captionpos=b,                    
	keepspaces=true,                 
	numbers=left,                    
	numbersep=5pt,                  
	showspaces=false,                
	showstringspaces=false,
	showtabs=false,                  
	tabsize=2
}

\definecolor{seagreen}{rgb}{0.18, 0.55, 0.34}
\definecolor{mediumviolet-red}{rgb}{0.78, 0.08, 0.52}
\definecolor{khaki}{rgb}{0.94, 0.9, 0.55}

\lstdefinelanguage{mypython}
{
	keywords=[1]{from, import, assert, not, print},
	keywordstyle=[1]{\color{mediumviolet-red}},
	keywords=[2]{surecr, torch, cp, lo, pl},
	keywordstyle=[2]{\color{seagreen}},
	numbers=none,
	upquote=true,
	showstringspaces=false,
	basicstyle=\ttfamily,
	columns=fullflexible,
	keepspaces=true,
	emph={True,False,as,def,return,float,class,match,switch,len},
	emphstyle={\color{seagreen}},
	frame=trBL,
	belowskip=1em,
	aboveskip=1em,
	captionpos=b
}

\usepackage[capitalize, nameinlink, noabbrev]{cleveref}
\crefname{equation}{}{}
\crefname{chapter}{Chapter}{Chapters}
\crefname{item}{item}{items}
\crefname{figure}{Figure}{Figures}
\crefname{theorem}{Theorem}{Theorems}
\crefname{assumption}{Assumption}{Assumption}
\crefname{lemma}{Lemma}{Lemmas}
\crefname{proposition}{Proposition}{Propositions}
\crefname{corollary}{Corollary}{Corollarys}
\crefname{definition}{Definition}{Definitions}
\crefname{fact}{Fact}{Facts}
\crefname{example}{Example}{Examples}
\crefname{algorithm}{Algorithm}{Algorithms}
\crefname{remark}{Remark}{Remarks}
\crefname{note}{Note}{Notes}
\crefname{notation}{Notation}{Notations}
\crefname{case}{Case}{Cases}
\crefname{exercise}{Exercise}{Exercises}
\crefname{question}{Question}{Questions}
\crefname{claim}{Claim}{Claims}
\crefname{enumi}{}{}

 \usepackage{textcomp}

\numberwithin{equation}{section}

\usepackage{amsmath,xparse}
\makeatletter
\NewDocumentCommand{\lplabel}{o m}{%
	\makebox[0pt][r]{#2\hspace*{2em}}%
	\IfNoValueF{#1}
	{\def\@currentlabel{#2}\ltx@label{#1}}
}
\makeatother

\theoremstyle{plain}
\newtheorem{theorem}{Theorem}[section]

\newtheorem{fact}[theorem]{Fact}
\newtheorem{lemma}[theorem]{Lemma}
\newtheorem{proposition}[theorem]{Proposition}

\newtheorem{assumption}[theorem]{Assumption}

\theoremstyle{definition}

\newtheorem{remark}[theorem]{Remark}

\usepackage{enumitem}

\newcommand{\minimize}{\ensuremath{\operatorname{minimize}}}
\newcommand{\maximize}{\ensuremath{\operatorname{maximize}}}
\newcommand{\argmin}{\ensuremath{\operatorname{argmin}}}

\newcommand{\weakly}{\ensuremath{{\;\operatorname{\rightharpoonup}\;}}}

\newcommand{\Prox}{\ensuremath{\operatorname{Prox}}}

\newcommand{\dom}{\ensuremath{\operatorname{dom}}}

\providecommand{\abs}[1]{\left|#1\right|}
\providecommand{\norm}[1]{\left\lVert#1\right\rVert}
\providecommand{\innp}[1]{\left\langle#1\right\rangle}

\newcommand\scalemath[2]{\scalebox{#1}{\mbox{\ensuremath{\displaystyle #2}}}}

\begin{document}

\title{Alternating Proximal Point Algorithm with Gradient Descent and Ascent Steps for 
Convex-Concave Saddle-Point Problems}

\author{
	Hui Ouyang\thanks{Department of Electrical Engineering, Stanford University.
		E-mail: \href{mailto:houyang@stanford.edu}{\texttt{houyang@stanford.edu}}.}
}

\date{October 30, 2023}

\maketitle

\begin{abstract}
Inspired by the Optimistic Gradient Ascent-Proximal Point Algorithm (OGAProx)
 proposed   by Bo{\c{t}}, Csetnek, and Sedlmayer for solving a saddle-point problem 
 associated with a convex-concave function 
 with a nonsmooth coupling function and one regularizing function, 
 we introduce the
Alternating Proximal Point Algorithm with Gradient Descent and Ascent Steps
for solving a saddle-point problem associated with a convex-concave function 
constructed by a smooth coupling function and two regularizing functions. 
In this work, we not only provide weak and linearly convergence of 
the sequence of iterations and of the minimax gap function evaluated at the ergodic 
sequences, similarly to what   Bo{\c{t}} et al.\,did,   but also demonstrate
the convergence and linearly convergence of function values evaluated at  convex 
combinations of iterations under convex and strongly convex assumptions, respectively.
\end{abstract}

{\small
	
}

 
\section{Introduction}

 In the whole work, $\mathcal{H}_{1}$ and $\mathcal{H}_{2}$ are real Hilbert spaces. 
 Let $X \subseteq \mathcal{H}_{1}$ and $Y \subseteq \mathcal{H}_{2}$ be  nonempty, 
 closed, and  convex subsets of the Hilbert spaces. 
 Let
 $f: \mathcal{H}_{1} \times \mathcal{H}_{2} \to \mathbf{R} \cup \{ -\infty, +\infty\}$ 
 satisfy that $(\forall y \in Y)$ $f(\cdot, y) : \mathcal{H}_{1}  \to \mathbf{R} \cup \{ 
 -\infty\}$ is proper, convex, and lower semicontinuous,  and that $(\forall x \in X)$
 $f(x, \cdot) : \mathcal{H}_{2} \to \mathbf{R} \cup \{ +\infty\}$ is proper, concave, and 
 upper semicontinuous.   
 We follow the definition of \emph{effective domain of $f$} from 
 \cite[Page~242]{RockafellarMinimax1964}: $\dom f:= \{ (\bar{x},\bar{y})  ~:~ (\forall y 
 \in 
 \mathcal{H}_{2}) f(\bar{x},y) < \infty \text{ and } (\forall x \in \mathcal{H}_{1}) f(x,\bar{y}) 
 > -\infty\}$.
 
 We say  $(x^{*},y^{*}) \in X \times Y$  is a \emph{saddle-point}  of $f$ if
\begin{align}\label{eq:saddle-point}
 	(\forall (x,y) \in X \times Y) \quad  f(x^{*},y) \leq  f(x^{*},y^{*}) \leq  f(x,y^{*}).
\end{align}
 Note that, via \cite[Pages~120 to 123]{RockafellarSaddlePoints1971}, because 
 $(\forall y \in Y)$ $f(\cdot, y) : \mathcal{H}_{1}  \to \mathbf{R} \cup \{ 
 -\infty\}$ is   lower semicontinuous and $(\forall x \in X)$
 $f(x, \cdot) : \mathcal{H}_{2} \to \mathbf{R} \cup \{ +\infty\}$ is  
 upper semicontinuous, if $X$ and $Y$ are closed and bounded, then the set of all 
 saddle-points of $f$ is nonempty, closed, convex, and bounded. 
 Moreover, when $ \mathcal{H}_{1} $ and $ \mathcal{H}_{2} $ are finite-dimensional, 
 we have the following result. 
 \begin{fact}  {\rm \cite[Chapter VII, Theorem~4.3.1]{HiriartJeanLemarechal1993}}
 	\label{fact:existencesaddlepoint}
 	Suppose that $\mathcal{H}_{1}$ and $\mathcal{H}_{2}$  are finite-dimensional. 
 	Suppose that $X \subseteq \mathcal{H}_{1}$ and $Y \subseteq \mathcal{H}_{2}$ are 
 	nonempty, closed, and convex, and that 
 	$(\forall y \in Y)$ $f(\cdot, y) : X  \to \mathbf{R} $ is  convex and $(\forall x \in X)$
 	$f(x, \cdot) : Y \to \mathbf{R}$ is  concave.
 	Suppose that $X$ is bounded, or that there exists $y_{0} \in Y$ such that 
 	$\lim\limits_{\stackrel{x \in X}{\norm{x} \to \infty}} f(x,y_{0}) =+ \infty$.
 	Suppose that $Y$ is bounded, or that there exists $x_{0} \in Y$ such that 
 	$\lim\limits_{\stackrel{y \in Y}{\norm{y} \to \infty}}  f(x_{0},y) =- \infty$. 
 	Then $f$ has a nonempty, convex, and compact set of saddle-points on $X \times Y$.
 \end{fact}

In the rest of this work, we assume that there exists at least one saddle-point of $f$, 
and we aim to solve the following  \emph{convex-concave saddle-point problem}:
\begin{align} \label{eq:problem}
 	\maximize_{y \in \mathcal{H}_{2}}	\minimize_{x \in \mathcal{H}_{1}} f(x,y)
\end{align}
 
  \subsection{Related Work}
 This work is motivated by \cite{BotCsetnekSedlmayer2022accelerated} by Bo{\c{t}},
 Csetnek, and Sedlmayer. 
 We compare \cite{BotCsetnekSedlmayer2022accelerated} and   this work 
 below. 
 \begin{enumerate}
 	\item The authors in \cite{BotCsetnekSedlmayer2022accelerated} work on the 
 	convex-concave function $ (\forall (x,y) \in \mathcal{H}_{1} \times \mathcal{H}_{2} )$ 
 	$f(x,y) =\Phi (x,y) -g(y)$, where $g$ is (strongly) convex and lower semicontinuous, 
 	$(\forall y \in \dom g)$
 	$\Phi(\cdot, y)$ is convex and lower semicontinuous, and for every $x$ in the 
 	projection
 	of $\dom \Phi$ onto $\mathcal{H}_{1}$, $\Phi(x, \cdot)$ is concave and Fr\'echet 
 	differentiable. 
 	
 	In this work, we consider the convex-concave function  
 	$(\forall (x,y) \in \mathcal{H}_{1} \times \mathcal{H}_{2} )$ $f(x,y) =f_{1}(x) +\Phi (x,y) 
 	-f_{2}(y)$, where $f_{1}$ and $f_{2}$ are (strongly) convex and lower 
 	semicontinuous,  
 	$(\forall y \in \dom f_{2})$ $\Phi(\cdot, y) $ is convex and Fr\'echet 
 	differentiable, 
 	and $(\forall x \in \dom f_{1})$ $\Phi(x,\cdot) $ is concave and  Fr\'echet 
 	differentiable.
 	
 	Because the associated convex-concave functions considered in this work is 
 	different from that worked in  
 	\cite{BotCsetnekSedlmayer2022accelerated}, our Alternating Proximal Point 
 	Algorithm with Gradient Descent and Ascent Steps  presented in 
 	\cref{eq:algorithmproxif1f2} below  is different from the Optimistic Gradient 
 	Ascent-Proximal Point Algorithm (OGAProx) studied in  	
 	\cite{BotCsetnekSedlmayer2022accelerated}.
 	
 	\item Let $((x^{k},y^{k}))_{k \in \mathbf{N}}$ be the associated sequence of 
 	iterations, 
 	and let $(x^{*},y^{*})$ be a saddle-point of $f$.
 	Although in  \cite{BotCsetnekSedlmayer2022accelerated}, the authors proved weak 
 	and linearly convergence of $((x^{k},y^{k}))_{k \in \mathbf{N}}$ to $(x^{*},y^{*})$, 
 	they didn't consider the convergence of function values to $f(x^{*},y^{*})$
 	except for a result on the minimax gap function evaluated at the ergodic sequences.
 	
 	In this work, we not only presented weak and linearly convergence of 
 	$((x^{k},y^{k}))_{k \in \mathbf{N}}$ to $(x^{*},y^{*})$,  but also show
 	the convergence $f(\hat{x}_{k},\hat{y}_{k}) \to f(x^{*},y^{*})$  under convex 
 	assumptions and  the
 	linearly convergence of $\left(f(\hat{x}_{k},\hat{y}_{k}) \right)_{k \in \mathbf{N}}$ to  
 	$  f(x^{*},y^{*})$ under strongly convex assumptions, where $(\forall k \in 
 	\mathbf{N})$ $t_{k} \in \mathbf{R}_{+}$ with $t_{0} \in \mathbf{R}_{++}$, 
 	$	(\forall k \in \mathbf{N})$ $
 	\hat{x}_{k} := \frac{1}{\sum^{k}_{i=0}t_{i}} 
 	\sum^{k}_{j=0}t_{j}x^{j+1} $ and $
 	\hat{y}_{k} := \frac{1}{\sum^{k}_{i=0}t_{i}} 
 	\sum^{k}_{j=0}t_{j}y^{j+1}$.
 \end{enumerate}
 
 \subsection{Outline}
 The work is organized as follows. 
  Some preliminary results are given in \cref{section:Preliminaries}. 
 We present the Alternating Proximal Point Algorithm with Gradient Descent and Ascent 
 Steps and convergence results of this algorithm in \cref{section:MainResults}.
 In particular, in \cref{subsection:convergence}, we prove   weak convergence of the 
 sequence of iterations 
 to a saddle-point of the associated convex-concave function.
  We also demonstrate the convergence of 
 function values evaluated at convex combinations of iterations to the function value of 
the saddle-point under certain convex assumptions of the two associated regularizing 
 functions. Additionally, we provide linear convergence results for both the sequence of 
 iterations and  function values evaluated at convex combinations of iterations under 
 specific strongly convex assumptions in \cref{subsection:LinearConvergence}.

 \section{Preliminaries} \label{section:Preliminaries}
 From now on, $\mathcal{H}$ is a real Hilbert space.
 \begin{fact} \label{fact:abc}
 	Let 
 	$a, b,$ and $c$
 	be points in $\mathcal{H}$. Then 
 	\[ 
 	\innp{a-b, c-b} = \frac{1}{2} \left(\norm{a-b}^{2} + \norm{c-b}^{2} 
 	-\norm{c-a}^{2}\right).
 	\]
 \end{fact}
 \begin{proof}
 	It is clear that
\begin{align*}
 		&\norm{a-b}^{2} + \norm{c-b}^{2} -\norm{c-a}^{2} \\
 		=& \norm{a}^{2} -2 \innp{a,b} +\norm{b}^{2} +\norm{c}^{2} 
 		-2 \innp{c,b} +\norm{b}^{2} -\norm{c}^{2} +2 \innp{c,a} -\norm{a}^{2} \\
 		=&  -2 \innp{a,b}+ 2 \innp{b,b} - 2 \innp{c,b} +2 \innp{c,a} \\
 		=& 2	\innp{a-b, c-b}.
\end{align*}
 \end{proof}

\begin{fact} { \rm \cite[Lemma~2.47]{BauschkCombettes2017}} 
\label{fact:weakconvege}
	Let $(x_{k})_{k \in \mathbf{N}}$ be a sequence in   
	$\mathcal{H}$ and let $C$ be a 
	nonempty subset of $\mathcal{H}$. 
	Suppose that   $( \forall x \in C)$ $ (\norm{x_{k} -x})_{k \in \mathbf{N}} $ converges 
	and  that every weak sequential cluster point of $(x_{k})_{k \in \mathbf{N}}$ 
	belongs to $C$. Then $(x_{k})_{k \in \mathbf{N}}$  converges weakly to a point in 
	$C$. 
\end{fact}

%

\begin{lemma} \label{lemma:fxkykx*y*}
	Let $f: \mathcal{H}_{1} \times \mathcal{H}_{2} \to \mathbf{R} \cup \{-\infty, +\infty\}$ 
	satisfy that $(\forall y \in \mathcal{H}_{2} )$ $f(\cdot, y)$ is convex 
	and $(\forall x \in \mathcal{H}_{1})$ $f(x,\cdot)$ is concave. 
	Let $(x^{*},y^{*})$ be a saddle-point of $f$, let $((x^{k},y^{k}))_{k \in \mathbf{N}}$ 
	be in $\mathcal{H}_{1}\times \mathcal{H}_{2}$, and let $(\forall k \in 
	\mathbf{N})$ $t_{k} \in \mathbf{R}_{+}$ with $t_{0} \in \mathbf{R}_{++}$.
	Set 
\begin{align} \label{eq:lemma:fxkykx*y*}
		(\forall k \in \mathbf{N} \smallsetminus \{0\}) \quad 
		\hat{x}_{k} := \frac{1}{\sum^{k-1}_{i=0}t_{i}} 
		\sum^{k-1}_{j=0}t_{j}x^{j+1} 
		\quad \text{and} \quad 
		\hat{y}_{k} := \frac{1}{\sum^{k-1}_{i=0}t_{i}} 
		\sum^{k-1}_{j=0}t_{j}y^{j+1}.
\end{align}
	Then we have that for every $k \in \mathbf{N}$,
	\begin{subequations}
		\begin{align}
			&f(\hat{x}_{k},\hat{y}_{k}) -f(x^{*},y^{*}) 
			\leq  \frac{1}{\sum^{k-1}_{i=0}t_{i}} 
			\sum^{k-1}_{j=0} t_{j} \left( f(x^{j+1}, \hat{y}_{k})  
			-f(x^{*},y^{j+1}) \right); \label{eq:lemma:fxkykx*y*:x}\\
			&f(x^{*},y^{*}) -f(\hat{x}_{k},\hat{y}_{k})
			\leq  \frac{1}{\sum^{k-1}_{i=0}t_{i}} 
			\sum^{k-1}_{j=0} t_{j} \left( f(x^{j+1},y^{*})  
			-f(\hat{x}_{k},y^{j+1}) \right).\label{eq:lemma:fxkykx*y*:y}
		\end{align}
	\end{subequations}
Consequently, if
\begin{align*}
	&\lim_{k \to \infty}   \frac{1}{\sum^{k-1}_{i=0}t_{i}} 
	\sum^{k-1}_{j=0}t_{j} \left( f(x^{j+1}, 
	\hat{y}_{k})  -f(x^{*},y^{j+1}) \right) =0, \text{ and}\\
	& \lim_{k \to \infty}  \frac{1}{\sum^{k-1}_{i=0}t_{i}} 
	\sum^{k-1}_{j=0}t_{j} \left( 
	f(x^{j+1},y^{*})  
	-f(\hat{x}_{k},y^{j+1}) \right)=0,
\end{align*} 
then 
$\lim_{k \to \infty} f(\hat{x}_{k},\hat{y}_{k}) = f(x^{*},y^{*})$.
\end{lemma}

\begin{proof}
	Let $k$ be in $\mathbf{N}$.
	Because $(\forall y \in \mathcal{H}_{2} )$ $f(\cdot, y)$ is convex 
	and $(\forall x \in \mathcal{H}_{1})$ $f(x,\cdot)$ is concave, we have that 
	for every $(x,y) \in  \mathcal{H}_{1} \times \mathcal{H}_{2}$,
	\begin{subequations} \label{eq:lemma:fxkykx*y*:cc}
		\begin{align}
			&f(\hat{x}_{k},y) \stackrel{\cref{eq:lemma:fxkykx*y*}}{=} 
			f\left( \frac{1}{\sum^{k-1}_{i=0}t_{i}} 
			\sum^{k-1}_{j=0} t_{j}x^{j+1},y \right) \leq
			 \frac{1}{\sum^{k-1}_{i=0}t_{i}} 
			\sum^{k-1}_{j=0}f(x^{j+1},y); \label{eq:lemma:fxkykx*y*:cc:x}\\
			&f(x,\hat{y}_{k})  \stackrel{\cref{eq:lemma:fxkykx*y*}}{=} 
			f\left(x,  \frac{1}{\sum^{k-1}_{i=0}t_{i}} 
			\sum^{k-1}_{j=0}t_{j}y^{j+1} \right) \geq 
			 \frac{1}{\sum^{k-1}_{i=0}t_{i}} 
			\sum^{k-1}_{j=0}t_{j}f(x,y^{j+1}). 
			\label{eq:lemma:fxkykx*y*:cc:y}
		\end{align}
	\end{subequations}
	Substitute $y$ and $x$ in \cref{eq:lemma:fxkykx*y*:cc:x} and 
	\cref{eq:lemma:fxkykx*y*:cc:y}, respectively, by $\hat{y}_{k}$ and $\hat{x}_{k}$
	to get that
	\begin{subequations} \label{eq:lemma:fxkykx*y*:xkyk}
		\begin{align}
			&-f(\hat{x}_{k},\hat{y}_{k})  \geq	 \frac{1}{\sum^{k-1}_{i=0}t_{i}} 
			\sum^{k-1}_{j=0}t_{j} \left( -f(x^{j+1},\hat{y}_{k}) \right);
			\label{eq:lemma:fxkykx*y*:xkyk:x}\\
			&f(\hat{x}_{k},\hat{y}_{k})  \geq 
			 \frac{1}{\sum^{k-1}_{i=0}t_{i}} 
			\sum^{k-1}_{j=0}t_{j}f(\hat{x}_{k},y^{j+1}). 
			\label{eq:lemma:fxkykx*y*:xkyk:y}
		\end{align}
	\end{subequations}
	Apply the fact that $(x^{*},y^{*})$ is a saddle-point of $f$ in the first inequalities of 
	the 	following two inequalities, and replace $y$ and $x$ in 
	\cref{eq:lemma:fxkykx*y*:cc:x} and 
	\cref{eq:lemma:fxkykx*y*:cc:y}, respectively, with $y^{*}$ and $x^{*}$  to obtain that
	\begin{subequations}
		\begin{align}
			&-f(x^{*},y^{*}) \geq -f(\hat{x}_{k},y^{*}) \geq   \frac{1}{\sum^{k-1}_{i=0}t_{i}} 
			\sum^{k-1}_{j=0}t_{j} \left( -f(x^{j+1},y^{*})\right);\label{eq:lemma:fxkykx*y*:x*}\\
			&f(x^{*},y^{*}) \geq f(x^{*},\hat{y}_{k})  \geq 	 \frac{1}{\sum^{k-1}_{i=0}t_{i}} 
			\sum^{k-1}_{j=0} t_{j}f(x^{*},y^{j+1}).\label{eq:lemma:fxkykx*y*:y*}
		\end{align}
	\end{subequations}
	Add \cref{eq:lemma:fxkykx*y*:xkyk:x} and \cref{eq:lemma:fxkykx*y*:y*} to get 
	\cref{eq:lemma:fxkykx*y*:x}. 
	Similarly, by adding \cref{eq:lemma:fxkykx*y*:xkyk:y} and 
	\cref{eq:lemma:fxkykx*y*:x*}, 
	we derive \cref{eq:lemma:fxkykx*y*:y}.
\end{proof}

 \section{Main Results} \label{section:MainResults}

\subsection{Settings}
Let $\mathcal{H}_{1}$ and $\mathcal{H}_{2}$ be real Hilbert spaces, 
let $\Phi: \mathcal{H}_{1} \times \mathcal{H}_{2} \to \mathbf{R} \cup \{- \infty, +\infty\}$
be a coupling function with 
$\dom \Phi := \{  (x,y) \in  \mathcal{H}_{1} \times \mathcal{H}_{2} ~:~ 
 \Phi(x,y) \in \mathbf{R} \} \neq \varnothing$,
 and let $(\forall i \in \{1,2\})$ $f_{i}: \mathcal{H}_{i} \to \mathbf{R} \cup \{+\infty\}$
 be regularizing functions. Define $f : \mathcal{H}_{1} \times \mathcal{H}_{2} \to 
 \mathbf{R} \cup \{-\infty, +\infty\}$ as 
\begin{align} \label{eq:def:f}
	 (\forall (x,y) \in \mathcal{H}_{1} \times \mathcal{H}_{2} ) \quad f(x,y) 
	 =f_{1}(x) +\Phi (x,y) -f_{2}(y).
\end{align}

\begin{assumption}\label{assumption}
	In the rest of this work, we have the following assumptions.
\begin{itemize}
	\item $(\forall i \in \{1,2\})$ $f_{i}$ is proper, lower semicontinuous, convex with 
	modulus $\mu_{i} \geq 0$, i.e., $f_{i} - \frac{\mu_{i}}{2} \norm{\cdot}^{2}$ is convex 
	$($note that when $\mu_{i} =0$, we have $f_{i}$ is convex, and that when 
	$\mu_{i} >0$, we have $f_{i}$ is $\mu_{i}$-strongly convex$)$.
	\item $\dom f_{1}$ and $\dom f_{2}$ are nonempty, closed, and convex. 
	\item $(\forall y \in \dom f_{2})$ $\Phi(\cdot, y) :\mathcal{H}_{1} \to \mathbf{R}  $ is 
	proper, convex, and Fr\'{e}chet differentiable.
	\item $(\forall x \in \dom f_{1})$ $\Phi(x,\cdot)  :\mathcal{H}_{2} \to \mathbf{R} $ is 
	proper, concave, and Fr\'{e}chet differentiable.
	\item There exist $L_{xx}$, $L_{xy}$, $L_{yx}$, and $L_{yy}$ in $\mathbf{R}_{+}$  
	such that for every $(x,y) \in \dom f_{1} \times  \dom f_{2}$ and $(x',y') \in \dom f_{1} 
	\times  \dom f_{2}$ such that 
	\begin{subequations}\label{eq:nablaPhixy}
		\begin{align}
			&\norm{\nabla_{x}\Phi(x,y) - \nabla_{x}\Phi(x',y')} \leq L_{xx}\norm{x -x'} 
			+L_{xy}\norm{y-y'}; \label{eq:nablaPhix}\\
			&\norm{\nabla_{y}\Phi(x,y) - \nabla_{y}\Phi(x',y')} \leq L_{yx}\norm{x -x'} 
			+L_{yy}\norm{y-y'}.\label{eq:nablaPhiy}
		\end{align}
	\end{subequations}
\end{itemize}
\end{assumption}

According to our assumptions above, the following statements are  equivalent. 
\begin{enumerate}
	\item $(x^{*},y^{*}) \in \mathcal{H}_{1} \times \mathcal{H}_{2}$  is a  saddle point  of 
	$f$.
	\item $(x^{*},y^{*}) \in \dom f_{1} \times \dom f_{2}$ and for every $(x,y) \in  
	\mathcal{H}_{1} \times \mathcal{H}_{2}$,
\[ 
		f_{1}(x^{*}) +\Phi(x^{*},y) -f_{2}(y) \leq f_{1}(x^{*}) +\Phi(x^{*},y^{*}) 
		-f_{2}(y^{*}) \leq f_{1}(x) +\Phi(x,y^{*}) -f_{2}(y^{*}).
\]
	\item $0 \in \partial f_{1}(x^{*}) +\nabla_{x}\Phi (x^{*},y^{*})$ and  
	$0 \in - \nabla_{y}\Phi (x^{*},y^{*}) + \partial f_{2}(x^{*})$. 
\end{enumerate}

We apply the convention that $+\infty - (+\infty) :=-\infty$ in this work. So we have that 
\begin{align*}
	f(x,y) =
	\begin{cases}
			f_{1}(x) +\Phi (x,y) -f_{2}(y) \quad &\text{if }x \in \dom f_{1}\text{ and } y \in \dom 
		f_{2},\\
		- \infty \quad &\text{if }y \notin \dom f_{2},\\
		+\infty \quad &\text{if }x \notin \dom f_{1} \text{ and } y \in \dom f_{2}.
	\end{cases}
\end{align*}
Similarly with the result $(2.6)$ on \cite[Page~243]{RockafellarMinimax1964}, we have 
$\dom f = \dom f_{1} \times \dom f_{2}$.
\begin{lemma} \label{lemma:TMM}
Define the operator	$T: \mathcal{H}_{1} \times \mathcal{H}_{2} \to 2^{\mathcal{H}_{1} 
\times 	\mathcal{H}_{2} }$  as 
\begin{align} \label{eq:lemma:TMM}
	(\forall (\bar{x},\bar{y}) \in \mathcal{H}_{1} \times \mathcal{H}_{2})  \quad
	T (\bar{x},\bar{y})  = \partial_{x}f(\bar{x},\bar{y}) \times \partial_{y}(-f(\bar{x},\bar{y})).
\end{align}
Then $T$ is maximally monotone. 
\end{lemma}

\begin{proof}	
	Set the operator $L:   \dom f_{1} \times \dom f_{2} \to \mathbf{R}$ as 
	$(\forall (x,y) \in \dom 
	f_{1} \times \dom f_{2})$ $L(x,y) =f_{1}(x) +\Phi (x,y) -f_{2}(y)$. 
	Based on our assumptions, $(\forall x \in \dom f_{1})$ 
	$L(x, \cdot)$ is upper semicontinuous and $(\forall y \in \dom f_{2})$ $L(\cdot, y)$
	is lower semicontinuous.  Hence,
	applying  \cite[Corollary~2]{RockafellarMonotone1970} with $Y= \mathcal{H}_{1}$,
	$Z=\mathcal{H}_{2}$, $C=\dom f_{1}$, $D=\dom f_{2}$, and $L=L$ defined above, 
	we obtain that  the operator $T$ defined in \cref{eq:lemma:TMM} is 
	maximally monotone.
\end{proof}

\subsection{Algorithm} 
We present our  Alternating Proximal Point 
Algorithm with Gradient Descent and Ascent Steps below.

Henceforth, let $(x^{0},y^{0}) \in \dom f_{1} \times \dom f_{2}$, 
and  let $(\tau_{k})_{k \in \mathbf{N}}$, 
$(\sigma_{k})_{k \in \mathbf{N}}$, 
$(\lambda_{k})_{k \in \mathbf{N}}$, 
and $(\theta_{k})_{k \in \mathbf{N}}$ be in $\mathbf{R}_{++}$. 
Set 
$x^{-1}=x^{0}$ and $y^{-1} =y^{0}$. 
In this work, we consider the sequence of iterations generated by the following 
scheme: for every $k \in \mathbf{N}$,
\begin{subequations}\label{eq:algorithmproxif1f2}
	\begin{align}
		&x^{k+1} = \Prox_{\tau_{k}f_{1}}
		\left(
		x^{k} -\tau_{k} 
		\left( (1+\lambda_{k}) \nabla_{x}\Phi(x^{k},y^{k}) 
		-\lambda_{k} \nabla_{x}\Phi(x^{k-1},y^{k-1})
		\right) \right);\label{eq:algorithmproxif1f2:x}\\ 
		&y^{k+1} = \Prox_{\sigma_{k}f_{2}}
		\left(
		y^{k} +\sigma_{k} 
		\left( (1+\theta_{k}) \nabla_{y}\Phi(x^{k},y^{k})
		-\theta_{k} \nabla_{y}\Phi(x^{k-1},y^{k-1})
		\right)\right). \label{eq:algorithmproxif1f2:y}
	\end{align}
\end{subequations}

Based on \cref{eq:algorithmproxif1f2}, the sequence $((x^{k},y^{k}))_{k \in 
\mathbf{N}}$ 
of iterates must be in $\dom f_{1} \times \dom f_{2}$.

\subsection{Auxiliary Results}
\begin{lemma}\label{lemma:proxf1f2ineq}
	Let $k \in \mathbf{N}$. We have the following results.
	\begin{enumerate}
		\item \label{lemma:proxf1f2ineq:x} $ \frac{x^{k} -x^{k+1}}{\tau_{k}}  -\lambda_{k} 
		\left(  
		\nabla_{x}\Phi(x^{k},y^{k}) 
		-\nabla_{x}\Phi(x^{k-1},y^{k-1})  \right) + \left( \nabla_{x} \Phi(x^{k+1},y^{k+1}) 
		-\nabla_{x} \Phi(x^{k},y^{k}) \right) \\
		 	\in  \partial f_{1}(x^{k+1}) +\nabla_{x} \Phi(x^{k+1},y^{k+1}) = \partial_{x} 
		 	f(x^{k+1},y^{k+1})$.
		\item \label{lemma:proxf1f2ineq:y} $\frac{y^{k} -y^{k+1}}{\sigma_{k}} 	+\theta_{k} 
		\left(  
		\nabla_{y}\Phi(x^{k},y^{k}) 
		-\nabla_{y}\Phi(x^{k-1},y^{k-1})  \right) 
		- \left( \nabla_{y} \Phi(x^{k+1},y^{k+1}) 
		-\nabla_{y} \Phi(x^{k},y^{k}) \right)\\
		   \in -\nabla_{y}\Phi(x^{k+1},y^{k+1})+\partial f_{2}(y^{k+1}) = \partial_{y} \left(- 
		   f(x^{k+1},y^{k+1}) \right)$.
		  \item \label{lemma:proxf1f2ineq:subdiffx}
		  $(\forall x \in \mathcal{H}_{1})$
		  $f_{1}(x) -f_{1}(x^{k+1}) \geq \frac{1}{2\tau_{k}} 
		  \left( \norm{x^{k}-x^{k+1}}^{2} + \norm{x-x^{k+1}}^{2} -\norm{x^{k}-x}^{2}    
		  \right)
		  + \\
		  \frac{\mu_{1}}{2} \norm{x-x^{k+1}}^{2}
		  -  
		  \innp{\lambda_{k} 
		  	\left(  
		  	\nabla_{x}\Phi(x^{k},y^{k}) 
		  	-\nabla_{x}\Phi(x^{k-1},y^{k-1})  \right)  
		  	+\nabla_{x} \Phi(x^{k},y^{k})  , x-x^{k+1}}$.
		  
		  \item \label{lemma:proxf1f2ineq:subdiffy}
		  $(\forall y \in \mathcal{H}_{2})$
		  $f_{2}(y) -f_{2}(y^{k+1}) \geq \frac{1}{2\sigma_{k}} \left( 
		  \norm{y^{k}-y^{k+1}}^{2} 
		  +\norm{y-y^{k+1}} -\norm{y^{k}-y}^{2} \right) +\\
		  \frac{\mu_{2}}{2}\norm{y- y^{k+1}}^{2} +
		   \innp{\theta_{k} 	\left(  
		  	\nabla_{y}\Phi(x^{k},y^{k}) 
		  	-\nabla_{y}\Phi(x^{k-1},y^{k-1})  \right) 
		  	+\nabla_{y} \Phi(x^{k},y^{k}) ,y-y^{k+1}  }.$
	\end{enumerate}
\end{lemma}

\begin{proof}
\cref{lemma:proxf1f2ineq:x}: 	According to the iterate scheme 
\cref{eq:algorithmproxif1f2:x} and the definition of 
	proximity mapping, we observe that
	\begin{align*}
	&x^{k+1} = \argmin_{x \in \mathcal{H}_{1}} f_{1}(x) + \frac{1}{2\tau_{k}} 
		\norm{x-  	\left(
			x^{k} -\tau_{k} 
			\left( (1+\lambda_{k}) \nabla_{x}\Phi(x^{k},y^{k}) 
			-\lambda_{k} \nabla_{x}\Phi(x^{k-1},y^{k-1})
			\right) \right) }^{2}\\
\Leftrightarrow &
0\in \partial f_{1}(x^{k+1})  +\frac{1}{\tau_{k}} 
\left(x^{k+1} -  \left(
x^{k} -\tau_{k} 
\left( (1+\lambda_{k}) \nabla_{x}\Phi(x^{k},y^{k})
-\lambda_{k} \nabla_{x}\Phi(x^{k-1},y^{k-1})
\right)    \right) \right) \\
\Leftrightarrow &
\frac{x^{k} -x^{k+1}}{\tau_{k}} -
\left( (1+\lambda_{k}) \nabla_{x}\Phi(x^{k},y^{k})
-\lambda_{k} \nabla_{x}\Phi(x^{k-1},y^{k-1})
\right)   \in \partial f_{1}(x^{k+1})\\
\Leftrightarrow &
\frac{x^{k} -x^{k+1}}{\tau_{k}}  -\lambda_{k} \left(  \nabla_{x}\Phi(x^{k},y^{k}) 
-\nabla_{x}\Phi(x^{k-1},y^{k-1})  \right) + \left( \nabla_{x} \Phi(x^{k+1},y^{k+1}) 
-\nabla_{x} \Phi(x^{k},y^{k}) \right) 
 	\\
 	& \quad \quad \quad \quad \quad \quad \in  \partial f_{1}(x^{k+1}) +\nabla_{x} 
 	\Phi(x^{k+1},y^{k+1}).
	\end{align*}
\cref{lemma:proxf1f2ineq:y}: In view of  the iterate scheme 
\cref{eq:algorithmproxif1f2:y} 
and the definition of proximity mapping, we have that
\begin{align*}
&	y^{k+1} =\argmin_{y \in \mathcal{H}_{2}} f_{2}(y) + \frac{1}{2 \sigma_{k}}
	\norm{y-    \left(
		y^{k} +\sigma_{k} 
		\left( (1+\theta_{k}) \nabla_{y}\Phi(x^{k},y^{k})
		-\theta_{k} \nabla_{y}\Phi(x^{k-1},y^{k-1})
		\right) \right) }^{2}\\
	\Leftrightarrow & 
	0 \in \partial f_{2} (y^{k+1}) +\frac{1}{\sigma_{k}} 
	\left(y^{k+1} -  \left(
	y^{k} +\sigma_{k} 
	\left( (1+\theta_{k}) \nabla_{y}\Phi(x^{k},y^{k})
	-\theta_{k} \nabla_{y}\Phi(x^{k-1},y^{k-1})
	\right) \right)   \right)\\
 \Leftrightarrow & 
		\frac{y^{k} -y^{k+1}}{\sigma_{k}} +
		\left( (1+\theta_{k}) \nabla_{y}\Phi(x^{k},y^{k})
		-\theta_{k} \nabla_{y}\Phi(x^{k-1},y^{k-1})
		\right)
		\in \partial f_{2}(y^{k+1})\\
 \Leftrightarrow & 
 \frac{y^{k} -y^{k+1}}{\sigma_{k}} 	+\theta_{k} \left(  \nabla_{y}\Phi(x^{k},y^{k}) 
 -\nabla_{y}\Phi(x^{k-1},y^{k-1})  \right) 
 - \left( \nabla_{y} \Phi(x^{k+1},y^{k+1}) 
 -\nabla_{y} \Phi(x^{k},y^{k}) \right)\\
 &\quad \quad \quad \quad \quad \quad \in -\nabla_{y}\Phi(x^{k+1},y^{k+1})+\partial 
 f_{2}(y^{k+1}).
\end{align*}
\cref{lemma:proxf1f2ineq:subdiffx}: Due to \cref{lemma:proxf1f2ineq:x}, we notice that 
\[
	\frac{x^{k} -x^{k+1}}{\tau_{k}}  -\lambda_{k} 
	\left(  
	\nabla_{x}\Phi(x^{k},y^{k}) 
	-\nabla_{x}\Phi(x^{k-1},y^{k-1})  \right)  
	-\nabla_{x} \Phi(x^{k},y^{k}) 
	\in  \partial f_{1}(x^{k+1}).  
 \]
Combine this with the assumption that $f_{1}$ is convex with modulus $\mu_{1}$ to 
derive that for every $x \in \mathcal{H}_{1}$,
\begin{align*}
	&	f_{1}(x) -f_{1}(x^{k+1}) \geq \frac{1}{\tau_{k}}\innp{x^{k}-x^{k+1},x-x^{k+1}}
	+\frac{\mu_{1}}{2} \norm{x-x^{k+1}}^{2} -\\
	& \innp{\lambda_{k} 
		\left(  
		\nabla_{x}\Phi(x^{k},y^{k}) 
		-\nabla_{x}\Phi(x^{k-1},y^{k-1})  \right)  
		+\nabla_{x} \Phi(x^{k},y^{k})  , x-x^{k+1}}\\
	\Leftrightarrow &
	f_{1}(x) -f_{1}(x^{k+1}) \geq \frac{1}{2\tau_{k}} 
	\left( \norm{x^{k}-x^{k+1}}^{2} + \norm{x-x^{k+1}}^{2} -\norm{x^{k}-x}^{2}    \right)
	+\frac{\mu_{1}}{2} \norm{x-x^{k+1}}^{2} -\\
	&
	 \innp{\lambda_{k} 
		\left(  
		\nabla_{x}\Phi(x^{k},y^{k}) 
		-\nabla_{x}\Phi(x^{k-1},y^{k-1})  \right)  
		+\nabla_{x} \Phi(x^{k},y^{k})  , x-x^{k+1}}.
\end{align*}
where in the last equivalence, we use \cref{fact:abc}.
 
\cref{lemma:proxf1f2ineq:subdiffy}: Based on \cref{lemma:proxf1f2ineq:y}, we have that
\[
	\frac{y^{k} -y^{k+1}}{\sigma_{k}} 	+\theta_{k} 
	\left(  
	\nabla_{y}\Phi(x^{k},y^{k}) 
	-\nabla_{y}\Phi(x^{k-1},y^{k-1})  \right) 
	+\nabla_{y} \Phi(x^{k},y^{k}) 
	\in \partial f_{2}(y^{k+1}).
 \]
Similarly with the proof of 
\cref{lemma:proxf1f2ineq:subdiffx} above, applying \cref{fact:abc} and the assumption
that $f_{2}$ is convex with modulus $\mu_{2}$, we derive that for every 
$y \in \mathcal{H}_{2}$,
\begin{align*}
	&f_{2}(y) -f_{2}(y^{k+1}) \geq \frac{1}{\sigma_{k}} \innp{y^{k} -y^{k+1}, y-y^{k+1}}
	+\frac{\mu_{2}}{2}\norm{y- y^{k+1}}^{2} +\\
	& \innp{\theta_{k} 	\left(  
		\nabla_{y}\Phi(x^{k},y^{k}) 
		-\nabla_{y}\Phi(x^{k-1},y^{k-1})  \right) 
		+\nabla_{y} \Phi(x^{k},y^{k}) ,y-y^{k+1}  }\\
	\Leftrightarrow &
	f_{2}(y) -f_{2}(y^{k+1}) \geq \frac{1}{2\sigma_{k}} \left( \norm{y^{k}-y^{k+1}}^{2} 
	+\norm{y-y^{k+1}}^{2} -\norm{y^{k}-y}^{2} \right) +
	\frac{\mu_{2}}{2}\norm{y- y^{k+1}}^{2} +\\
	& \innp{\theta_{k} 	\left(  
		\nabla_{y}\Phi(x^{k},y^{k}) 
		-\nabla_{y}\Phi(x^{k-1},y^{k-1})  \right) 
		+\nabla_{y} \Phi(x^{k},y^{k}) ,y-y^{k+1}  }.
\end{align*}
\end{proof}

\begin{lemma}\label{lemma:proxf1f2preliminary}
	Let $(\bar{x}, \bar{y}) \in \dom f_{1} \times \dom f_{2}$
	and let $(\hat{x},\hat{y}) \in\dom f_{1} \times \dom f_{2}$. 
	We have the following statements. 
	\begin{enumerate}
		\item \label{lemma:proxf1f2preliminary:diff} 
		$( \forall (x,y)\in 
		\mathcal{H}_{1}\times \mathcal{H}_{2} )$
		$\innp{\nabla_{x} 
			\Phi(\bar{x},\bar{y}), x-\bar{x}}  \leq \Phi(x, \bar{y}) -\Phi(\bar{x},\bar{y}) $
		and  $\innp{-\nabla_{y}\Phi(\bar{x},\bar{y}),y-\bar{y}}  \leq -\Phi(\bar{x},y) 
		+\Phi(\bar{x},\bar{y})$.
		\item \label{lemma:proxf1f2preliminary:Lipschitzx}
		Let $\alpha$ and $\beta$ be in $\mathbf{R}_{++}$.
		Then for every $x	\in \mathcal{H}_{1}$,
\begin{align*}
	&\abs{\innp{\nabla_{x} \Phi(\bar{x},\bar{y}) - 
			\nabla_{x}\Phi(\hat{x},\hat{y}),x-\bar{x}}}\\
		 \leq & \frac{L_{xx}}{2} \left( \alpha 
	\norm{x 
		-\bar{x}}^{2} +\frac{1}{\alpha} \norm{\bar{x} 
		-\hat{x}}^{2}   \right) + \frac{L_{xy}}{2} \left( \beta \norm{x -\bar{x}}^{2} 
	+\frac{1}{\beta} 
	\norm{\bar{y} -\hat{y}}^{2}  \right).
\end{align*}
		\item \label{lemma:proxf1f2preliminary:Lipschitzy}
		 Let  $\gamma$ and $\delta$ be in $\mathbf{R}_{++}$.
	Then for every $y \in \mathcal{H}_{2}$,
\begin{align*}
	&\abs{ \innp{\nabla_{y} \Phi(\bar{x},\bar{y}) - 
			\nabla_{y}\Phi(\hat{x},\hat{y}),y-\bar{y}}  } \\
		\leq & \frac{L_{yx}}{2} \left( \gamma 
	\norm{y -\bar{y}}^{2} + \frac{1}{\gamma} 
	\norm{\bar{x} -\hat{x}}^{2} \right)  
	+ \frac{L_{yy}}{2}  \left( \delta \norm{y -\bar{y}}^{2} +\frac{1}{\delta}\norm{\bar{y} 
		-\hat{y}}^{2}  \right).
\end{align*}
	\end{enumerate}
\end{lemma}

\begin{proof}
	\cref{lemma:proxf1f2preliminary:diff}: Because $\Phi(\cdot, \bar{y})$ and 
	$-\Phi(\bar{x}, \cdot)$ are convex, we know that for every 	 $ (x,y)\in 
	\mathcal{H}_{1}\times \mathcal{H}_{2} $,
\[
	 \innp{\nabla_{x} 	\Phi(\bar{x},\bar{y}), x-\bar{x}}  \leq \Phi(x, \bar{y}) 
	 -\Phi(\bar{x},\bar{y})  
		\text{ and }  \innp{-\nabla_{y}\Phi(\bar{x},\bar{y}),y-\bar{y}}  \leq -\Phi(\bar{x},y) 
		+\Phi(\bar{x},\bar{y}).
 \]

	\cref{lemma:proxf1f2preliminary:Lipschitzx}: 
	Based on the Cauchy-Schwarz inequality 
	and our assumption 	\cref{eq:nablaPhix}, we observe that
	\begin{align*}
		&\abs{\innp{\nabla_{x} \Phi(\bar{x},\bar{y}) - 
		\nabla_{x}\Phi(\hat{x},\hat{y}),x-\bar{x}}} \\
		\leq & \norm{\nabla_{x} \Phi(\bar{x},\bar{y}) 
		- \nabla_{x}\Phi(\hat{x},\hat{y})}\norm{x-\bar{x}}\\
	 \leq & \left(L_{xx}\norm{\bar{x} -\hat{x}} 	+L_{xy}\norm{\bar{y} 
	-\hat{y}}\right)\norm{x-\bar{x}} \\
 \leq & \frac{L_{xx}}{2} \left( \alpha \norm{x -\bar{x}}^{2} +\frac{1}{\alpha} \norm{\bar{x} 
-\hat{x}}^{2}   \right) + \frac{L_{xy}}{2} \left( \beta \norm{x -\bar{x}}^{2} +\frac{1}{\beta} 
\norm{\bar{y} -\hat{y}}^{2}  \right).
	\end{align*}

\cref{lemma:proxf1f2preliminary:Lipschitzy}: Similarly, by the Cauchy-Schwarz 
inequality 
and our assumption 	\cref{eq:nablaPhiy}, we have that
\begin{align*}
	&\abs{ \innp{\nabla_{y} \Phi(\bar{x},\bar{y}) - 
			\nabla_{y}\Phi(\hat{x},\hat{y}),y-\bar{y}}  }\\
\leq & \norm{\nabla_{y} \Phi(\bar{x},\bar{y}) - 
		\nabla_{y}\Phi(\hat{x},\hat{y})} \norm{y -\bar{y}}\\
\leq & 	\left(L_{yx}\norm{\bar{x} -\hat{x}} 	+L_{yy}\norm{\bar{y} 	-\hat{y}}\right)
\norm{y -\bar{y}}\\
\leq & \frac{L_{yx}}{2} \left( \gamma \norm{y -\bar{y}}^{2} + \frac{1}{\gamma} 
\norm{\bar{x} -\hat{x}}^{2} \right)  
+ \frac{L_{yy}}{2}  \left( \delta \norm{y -\bar{y}}^{2} +\frac{1}{\delta}\norm{\bar{y} 
-\hat{y}}^{2}  \right).
\end{align*}
\end{proof}

\begin{proposition}\label{prop:fxk+1yxyk+1}
	Let $(x,y)$ be in $\mathcal{H}_{1} \times \mathcal{H}_{2}$, and let $(\alpha_{k})_{k 
	\in \mathbf{N}}$, $(\beta_{k})_{k \in \mathbf{N}}$, $(\gamma_{k})_{k \in 
	\mathbf{N}}$, $(\delta_{k})_{k \in \mathbf{N}}$ be in $\mathbf{R}_{++}$. Let $k$ be 
	in 	$\mathbf{N}$. Then
\[ 
		f(x^{k+1},y) -f(x,y^{k+1}) \leq a_{k}(x,y) -b_{k+1}(x,y) -c_{k},
\]
where 
\begin{subequations}\label{lemma:fxk+1yxyk+1:ak}
	\begin{align}
		&a_{k}(x,y) \\
		:=&  \frac{1}{2\tau_{k}} \norm{x-x^{k}}^{2} 
		+\frac{1}{2\sigma_{k}}\norm{y-y^{k}}^{2} \\
		&+\left( \frac{\lambda_{k}L_{xx}}{2\alpha_{k}} + 
		\frac{\theta_{k}L_{yx}}{2\gamma_{k}} \right) \norm{x^{k}-x^{k-1}}^{2} 
		+\left(\frac{\lambda_{k}L_{xy}}{2\beta_{k}}+\frac{\theta_{k}L_{yy}}{2\delta_{k}} 
		\right) 
		\norm{y^{k}-y^{k-1}}^{2}\\
		&+\scalemath{0.95}{ \lambda_{k} \innp{\nabla_{x}\Phi(x^{k},y^{k}) 
			-\nabla_{x}\Phi(x^{k-1},y^{k-1}),x-x^{k}}   -\theta_{k} \innp{	 
			\nabla_{y}\Phi(x^{k},y^{k}) 
			-\nabla_{y}\Phi(x^{k-1},y^{k-1}),y-y^{k}  } },
	\end{align}
\end{subequations}
\begin{subequations}\label{lemma:fxk+1yxyk+1:bk}
	\begin{align}
		&b_{k+1}(x,y)\\
		:=& \frac{1}{2} \left(\frac{1}{\tau_{k}}+\mu_{1}\right)
		\norm{x-x^{k+1}}^{2}  
		+\frac{1}{2} \left(\frac{1}{\sigma_{k}} +\mu_{2} \right)\norm{y-y^{k+1}}^{2} 
		\\
		&+\left( \frac{ L_{xx}}{2\alpha_{k+1}} + 
		\frac{ L_{yx}}{2\gamma_{k+1}} \right) \norm{x^{k+1}-x^{k}}^{2} 
		+\left(\frac{L_{xy}}{2\beta_{k+1}}+\frac{L_{yy}}{2\delta_{k+1}}  \right) 
		\norm{y^{k+1}-y^{k}}^{2}\\
		& +\innp{\nabla_{x}\Phi 	(x^{k+1},y^{k+1})-\nabla_{x} \Phi(x^{k},y^{k}), x-x^{k+1}} 
		-\innp{\nabla_{y}\Phi(x^{k+1},y^{k+1})-\nabla_{y} \Phi(x^{k},y^{k}), y-y^{k+1} },
	\end{align}
\end{subequations}
and
\begin{subequations}\label{lemma:fxk+1yxyk+1:ck}
	\begin{align}
		c_{k}:=	 
		&\frac{1}{2}\left(\frac{1}{\tau_{k}}- \lambda_{k} \alpha_{k}L_{xx} -
		\lambda_{k}\beta_{k}L_{xy}- \frac{ L_{xx}}{ \alpha_{k+1}} -
		\frac{ L_{yx}}{ \gamma_{k+1}}  \right) \norm{x^{k+1}-x^{k}}^{2} \\
		&+\frac{1}{2 } \left(\frac{1}{\sigma_{k}} -\theta_{k}\gamma_{k}L_{yx} 
		-\theta_{k}\delta_{k}L_{yy}
		-\frac{L_{xy}}{\beta_{k+1}}-\frac{L_{yy}}{\delta_{k+1}}  
		\right) 
		\norm{y^{k+1}-y^{k}}^{2}.
	\end{align}
\end{subequations}
\end{proposition}

\begin{proof}
	Applying \cref{lemma:proxf1f2preliminary}\cref{lemma:proxf1f2preliminary:diff} with 
	$(\bar{x},\bar{y}) =(x^{k+1},y^{k+1}) $, we 
	observe that
	\begin{subequations}\label{eq:lemma:fxk+1yxyk+1}
		\begin{align}
			&\Phi (x^{k+1},y) - \Phi (x, y^{k+1}) \\
			=& \Phi (x^{k+1},y) -\Phi (x^{k+1},y^{k+1}) +  \Phi 
			(x^{k+1},y^{k+1}) - \Phi (x, y^{k+1})\\
			\leq &\innp{\nabla_{y}\Phi(x^{k+1},y^{k+1}), y-y^{k+1} } -\innp{\nabla_{x}\Phi 
				(x^{k+1},y^{k+1}), x-x^{k+1}}.
		\end{align}  
	\end{subequations}
Employing \cref{lemma:proxf1f2preliminary}\cref{lemma:proxf1f2preliminary:Lipschitzx} 
with $(\bar{x},\bar{y}) = (x^{k},y^{k})$, $(\hat{x},\hat{y}) = (x^{k-1},y^{k-1})$, $x=x^{k
+1}$, and 
$(\alpha, \beta) =(\alpha_{k},\beta_{k})$, we 
have that 
\begin{subequations}\label{eq:lemma:nablax}
	\begin{align}
	&  \innp{\nabla_{x}\Phi(x^{k},y^{k})-\nabla_{x}\Phi(x^{k-1},y^{k-1}),x^{k}-x^{k+1}}  \\
	\leq & \frac{L_{xx}}{2} \left( \alpha_{k} 
	\norm{x^{k+1} -x^{k}}^{2} +\frac{1}{\alpha_{k}} \norm{x^{k}
		-x^{k-1}}^{2}   \right) + \frac{L_{xy}}{2} \left( \beta_{k} \norm{x^{k+1} -x^{k}}^{2} 
	+\frac{1}{\beta_{k}} 
	\norm{y^{k} -y^{k-1}}^{2}  \right). 
	\end{align}
\end{subequations}
Similarly, invoke 
\cref{lemma:proxf1f2preliminary}\cref{lemma:proxf1f2preliminary:Lipschitzy} with  
$(\bar{x},\bar{y}) = (x^{k},y^{k})$, $(\hat{x},\hat{y}) = (x^{k-1},y^{k-1})$, $y=y^{k+1}$,
and $(\gamma, \delta)=(\gamma_{k}, \delta_{k})$ to get that
\begin{subequations}\label{eq:lemma:nablay}
	\begin{align}
	&  -\innp{\nabla_{y}\Phi(x^{k},y^{k}) -\nabla_{y}\Phi(x^{k-1},y^{k-1}),y^{k}-y^{k+1}  }\\
	\leq &  \frac{L_{yx}}{2} \left( \gamma_{k} 
	\norm{y^{k+1} -y^{k}}^{2} + \frac{1}{\gamma_{k}} 
	\norm{x^{k} -x^{k-1}}^{2} \right)  
	+ \frac{L_{yy}}{2}  \left( \delta_{k} \norm{y^{k+1} -y^{k}}^{2} 
	+\frac{1}{\delta_{k}}\norm{y^{k}-y^{k-1}}^{2}  \right).
	\end{align}
\end{subequations}
	Applying \cref{eq:def:f} in the first equality, \cref{eq:lemma:fxk+1yxyk+1} in the first 
	inequality,  
	\cref{lemma:proxf1f2ineq}\cref{lemma:proxf1f2ineq:subdiffx}$\&$\cref{lemma:proxf1f2ineq:subdiffy}
in the second inequality, and \cref{eq:lemma:nablax}$\&$\cref{eq:lemma:nablay} in the 
third inequality below, we have that
\begin{align*}
		&f(x^{k+1},y) -f(x,y^{k+1})\\
=&f_{1} (x^{k+1}) +\Phi(x^{k+1},y) -f_{2}(y) -f_{1}(x) - \Phi(x,y^{k+1})+f_{2}(y^{k+1})\\
\leq & \scalemath{0.95}{ \innp{\nabla_{y}\Phi(x^{k+1},y^{k+1}), y-y^{k+1} } 
-\innp{\nabla_{x}\Phi 
	(x^{k+1},y^{k+1}), x-x^{k+1}} + f_{1} (x^{k+1})  - f_{1}(x)+f_{2}(y^{k+1})-f_{2}(y) }\\
\leq & \scalemath{0.95}{  \innp{\nabla_{y}\Phi(x^{k+1},y^{k+1}), y-y^{k+1} } 
-\innp{\nabla_{x}\Phi 
	(x^{k+1},y^{k+1}), x-x^{k+1}} -\frac{\mu_{1}}{2} \norm{x-x^{k+1}}^{2} 
	-\frac{\mu_{2}}{2} \norm{y-y^{k+1}}^{2} } \\
&+ \scalemath{0.9}{\frac{1}{2\tau_{k}} 
\left( \norm{x-x^{k}}^{2}  - \norm{x-x^{k+1}}^{2} - \norm{x^{k}-x^{k+1}}^{2} \right)
+\frac{1}{2\sigma_{k}} \left( 
\norm{y-y^{k}}^{2}-\norm{y-y^{k+1}}^{2}  -\norm{y^{k}-y^{k+1}}^{2}  \right)} \\
&+\lambda_{k} \innp{\nabla_{x}\Phi(x^{k},y^{k}) 
	-\nabla_{x}\Phi(x^{k-1},y^{k-1}),x-x^{k+1}} +\innp{\nabla_{x} \Phi(x^{k},y^{k}), 
	x-x^{k+1}} \\
&-\theta_{k} \innp{	 
	\nabla_{y}\Phi(x^{k},y^{k}) 
	-\nabla_{y}\Phi(x^{k-1},y^{k-1}),y-y^{k+1}  } -\innp{\nabla_{y} \Phi(x^{k},y^{k}) 
	,y-y^{k+1}}\\
=&\innp{\nabla_{y}\Phi(x^{k+1},y^{k+1})-\nabla_{y} \Phi(x^{k},y^{k}), y-y^{k+1} } 
-\innp{\nabla_{x}\Phi 	(x^{k+1},y^{k+1})-\nabla_{x} \Phi(x^{k},y^{k}), x-x^{k+1}} \\
&+\scalemath{0.9}{\frac{1}{2\tau_{k}} 
\left( \norm{x-x^{k}}^{2}  - \norm{x-x^{k+1}}^{2} - \norm{x^{k}-x^{k+1}}^{2} \right)
+\frac{1}{2\sigma_{k}} \left( 
\norm{y-y^{k}}^{2}-\norm{y-y^{k+1}}^{2}  -\norm{y^{k}-y^{k+1}}^{2}  \right)} \\
&+\lambda_{k} \innp{\nabla_{x}\Phi(x^{k},y^{k}) 
	-\nabla_{x}\Phi(x^{k-1},y^{k-1}),x-x^{k}}   -\theta_{k} \innp{	 
	\nabla_{y}\Phi(x^{k},y^{k}) 
	-\nabla_{y}\Phi(x^{k-1},y^{k-1}),y-y^{k}  } \\
&+\scalemath{0.9}{\lambda_{k} \innp{\nabla_{x}\Phi(x^{k},y^{k}) 
	-\nabla_{x}\Phi(x^{k-1},y^{k-1}),x^{k}-x^{k+1}}   -\theta_{k} \innp{	 
	\nabla_{y}\Phi(x^{k},y^{k}) 
	-\nabla_{y}\Phi(x^{k-1},y^{k-1}),y^{k}-y^{k+1}  } }\\
&-\frac{\mu_{1}}{2} \norm{x-x^{k+1}}^{2} 
-\frac{\mu_{2}}{2} \norm{y-y^{k+1}}^{2}  \\
\leq &\innp{\nabla_{y}\Phi(x^{k+1},y^{k+1})-\nabla_{y} \Phi(x^{k},y^{k}), y-y^{k+1} } 
-\innp{\nabla_{x}\Phi 	(x^{k+1},y^{k+1})-\nabla_{x} \Phi(x^{k},y^{k}), x-x^{k+1}} \\
&+\scalemath{0.9}{\frac{1}{2\tau_{k}} 
\left( \norm{x-x^{k}}^{2}  - \norm{x-x^{k+1}}^{2} - \norm{x^{k}-x^{k+1}}^{2} \right)
+\frac{1}{2\sigma_{k}} \left( 
\norm{y-y^{k}}^{2}-\norm{y-y^{k+1}}^{2}  -\norm{y^{k}-y^{k+1}}^{2}  \right)} \\
&+\scalemath{0.95}{ \lambda_{k} \innp{\nabla_{x}\Phi(x^{k},y^{k}) 
	-\nabla_{x}\Phi(x^{k-1},y^{k-1}),x-x^{k}}   -\theta_{k} \innp{	 
	\nabla_{y}\Phi(x^{k},y^{k}) 
	-\nabla_{y}\Phi(x^{k-1},y^{k-1}),y-y^{k}  }} \\
&+\scalemath{0.95}{  \frac{\lambda_{k}L_{xx}}{2} \left( \alpha_{k} 
\norm{x^{k+1} -x^{k}}^{2} +\frac{1}{\alpha_{k}} \norm{x^{k}
	-x^{k-1}}^{2}   \right) + \frac{\lambda_{k}L_{xy}}{2} \left( \beta_{k} \norm{x^{k+1} 
	-x^{k}}^{2} 
+\frac{1}{\beta_{k}} 
\norm{y^{k} -y^{k-1}}^{2}  \right) }\\
&+\scalemath{0.95}{  \frac{\theta_{k}L_{yx}}{2} \left( \gamma_{k} 
\norm{y^{k+1} -y^{k}}^{2} + \frac{1}{\gamma_{k}} 
\norm{x^{k} -x^{k-1}}^{2} \right)  
+ \frac{\theta_{k}L_{yy}}{2}  \left( \delta_{k} \norm{y^{k+1} -y^{k}}^{2} 
+\frac{1}{\delta_{k}}\norm{y^{k}-y^{k-1}}^{2}  \right) }\\
&-\frac{\mu_{1}}{2} \norm{x-x^{k+1}}^{2} 
-\frac{\mu_{2}}{2} \norm{y-y^{k+1}}^{2}  \\
&=a_{k}(x,y) -b_{k+1}(x,y) -c_{k},
\end{align*}
where $a_{k}(x,y)$, $b_{k+1}(x,y)$, and $c_{k}$ are defined in 
\cref{lemma:fxk+1yxyk+1:ak}, 
\cref{lemma:fxk+1yxyk+1:bk}, and \cref{lemma:fxk+1yxyk+1:ck}, respectively. 
\end{proof}

\begin{lemma} \label{lemma:akgeq}
	Let $(x,y)$ be in $\mathcal{H}_{1} \times \mathcal{H}_{2}$ and let $k$ be in 
	$\mathbf{N}$.
	Consider  $a_{k}(x,y)$ defined as \cref{lemma:fxk+1yxyk+1:ak}. We have that
\[ 
		a_{k}(x,y) \geq  \frac{1}{2}  \left( \frac{1}{\tau_{k}}  - L_{xx}\lambda_{k}\alpha_{k} 
		-L_{xy}\lambda_{k}\beta_{k}   
		\right)\norm{x-x^{k}}^{2} 
		+\frac{1}{2} \left( \frac{1}{\sigma_{k}} - L_{yx}\theta_{k}\gamma_{k}
		-L_{yy}\theta_{k}\delta_{k} \right)\norm{y-y^{k}}^{2}. 
\]
\end{lemma}

\begin{proof}
	Applying 
	\cref{lemma:proxf1f2preliminary}\cref{lemma:proxf1f2preliminary:Lipschitzx}
	with $(\bar{x},\bar{y}) =(x^{k},y^{k})$, $(\hat{x},\hat{y}) = (x^{k-1},y^{k-1})$, and 
	$(\alpha,\beta) = (\alpha_{k},\beta_{k})$,  we observe that
	\begin{subequations}\label{eq:prop:proxproxdiff:nablax}
		\begin{align}
			&\innp{\nabla_{x}\Phi(x^{k},y^{k}) 	-\nabla_{x}\Phi(x^{k-1},y^{k-1}),x-x^{k}} \\
			\geq &  -\frac{L_{xx}}{2} \left( \alpha_{k}  \norm{x 	-x^{k}}^{2} 
			+\frac{1}{\alpha_{k}} 
			\norm{x^{k}
				-x^{k-1}}^{2}   \right) 
			- \frac{L_{xy}}{2} \left( \beta_{k} \norm{x -x^{k}}^{2} 
			+\frac{1}{\beta_{k}} 
			\norm{y^{k} -y^{k-1}}^{2}  \right).
		\end{align}
	\end{subequations}
	Similarly, employing   
	\cref{lemma:proxf1f2preliminary}\cref{lemma:proxf1f2preliminary:Lipschitzy}
	with $(\bar{x},\bar{y}) =(x^{k},y^{k})$, $(\hat{x},\hat{y}) = (x^{k-1},y^{k-1})$, 
	and $(\gamma, \delta) = (\gamma_{k},\delta_{k})$,  we have that 
	\begin{subequations}\label{eq:prop:proxproxdiff:nablay}
		\begin{align}
			&\innp{	\nabla_{y}\Phi(x^{k},y^{k}) -\nabla_{y}\Phi(x^{k-1},y^{k-1}),y-y^{k}  }\\
			\geq &  -\frac{L_{yx}}{2} \left( \gamma_{k} 
			\norm{y -y^{k}}^{2} + \frac{1}{\gamma_{k}} 
			\norm{x^{k} -x^{k-1}}^{2} \right)  
			- \frac{L_{yy}}{2}  \left( \delta_{k} \norm{y -y^{k}}^{2} 
			+\frac{1}{\delta_{k}}\norm{y^{k} 
				-y^{k-1}}^{2}  \right).
		\end{align}
	\end{subequations}
Applying \cref{eq:prop:proxproxdiff:nablax} and 
	\cref{eq:prop:proxproxdiff:nablay} in the inequality below, we have that
	\begin{align*}
		&a_{k}(x,y) \\
		=&  \frac{1}{2\tau_{k}} \norm{x-x^{k}}^{2} 
		+\frac{1}{2\sigma_{k}}\norm{y-y^{k}}^{2} \\
		&+\left( \frac{\lambda_{k}L_{xx}}{2\alpha_{k}} + 
		\frac{\theta_{k}L_{yx}}{2\gamma_{k}} \right) \norm{x^{k}-x^{k-1}}^{2} 
		+\left(\frac{\lambda_{k}L_{xy}}{2\beta_{k}}+\frac{\theta_{k}L_{yy}}{2\delta_{k}} 
		\right) 
		\norm{y^{k}-y^{k-1}}^{2}\\
		&+\lambda_{k} \innp{\nabla_{x}\Phi(x^{k},y^{k}) 
			-\nabla_{x}\Phi(x^{k-1},y^{k-1}),x-x^{k}} 
		-\theta_{k} \innp{	 
			\nabla_{y}\Phi(x^{k},y^{k}) 
			-\nabla_{y}\Phi(x^{k-1},y^{k-1}),y-y^{k}  }\\
		\geq & \frac{1}{2\tau_{k}} \norm{x-x^{k}}^{2} 
		+\frac{1}{2\sigma_{k}}\norm{y-y^{k}}^{2} \\
		&+\left( \frac{\lambda_{k}L_{xx}}{2\alpha_{k}} + 
		\frac{\theta_{k}L_{yx}}{2\gamma_{k}} \right) \norm{x^{k}-x^{k-1}}^{2} 
		+\left(\frac{\lambda_{k}L_{xy}}{2\beta_{k}}+\frac{\theta_{k}L_{yy}}{2\delta_{k}} 
		\right) 
		\norm{y^{k}-y^{k-1}}^{2}\\
		&-\frac{L_{xx}\lambda_{k}}{2} \left( \alpha_{k}  \norm{x 	-x^{k}}^{2} 
		+\frac{1}{\alpha_{k}} 
		\norm{x^{k}
			-x^{k-1}}^{2}   \right) 
		- \frac{L_{xy}\lambda_{k}}{2} \left( \beta_{k} \norm{x -x^{k}}^{2} 
		+\frac{1}{\beta_{k}} 
		\norm{y^{k} -y^{k-1}}^{2}  \right)\\
		&-\frac{L_{yx}\theta_{k}}{2} \left( \gamma_{k} 
		\norm{y -y^{k}}^{2} + \frac{1}{\gamma_{k}} 
		\norm{x^{k} -x^{k-1}}^{2} \right)  
		- \frac{L_{yy}\theta_{k}}{2}  \left( \delta_{k} \norm{y -y^{k}}^{2} 
		+\frac{1}{\delta_{k}}\norm{y^{k} 
			-y^{k-1}}^{2}  \right)\\
		=&  \frac{1}{2}  \left( \frac{1}{\tau_{k}}  - L_{xx}\lambda_{k}\alpha_{k} 
		-L_{xy}\lambda_{k}\beta_{k}   
		\right)\norm{x-x^{k}}^{2} 
		+\frac{1}{2} \left( \frac{1}{\sigma_{k}} - L_{yx}\theta_{k}\gamma_{k}
		-L_{yy}\theta_{k}\delta_{k} \right)\norm{y-y^{k}}^{2}.
	\end{align*}
\end{proof}

From now on,  let $(t_{k})_{k \in \mathbf{N}}$ be in $\mathbf{R}_{+}$ with $t_{0} \in 
\mathbf{R}_{++}$,
let $(\alpha_{k})_{k \in \mathbf{N}}$, $(\beta_{k})_{k \in \mathbf{N}}$,
$(\gamma_{k})_{k \in \mathbf{N}}$, and $(\delta_{k})_{k \in \mathbf{N}}$
be in $\mathbf{R}_{++}$,
and set 
\[
	(\forall K \in \mathbf{N} \smallsetminus \{0\}) \quad \hat{x}_{K}:= \frac{1}{ 
	\sum^{K-1}_{i=0}t_{i}} 
	\sum^{K-1}_{k=0} t_{k}x^{k+1}  
	\quad \text{and} \quad 
	\hat{y}_{K}:= \frac{1}{ \sum^{K-1}_{i=0}t_{i}} \sum^{K-1}_{k=0} t_{k}y^{k+1}.
 \]

\begin{lemma} \label{lemma:abck}
	Let $k \in \mathbf{N}$ and let $(x,y) \in \mathcal{H}_{1} \times \mathcal{H}_{2}$. We 
	have the following results.
	\begin{enumerate}
		\item \label{lemma:abck:bk} Suppose that 
		$t_{k}\left(\frac{1}{\tau_{k}}+\mu_{1}\right)
		\geq  \frac{t_{k+1}}{\tau_{k+1}}$,
		$t_{k}\left(\frac{1}{\sigma_{k}} +\mu_{2} \right) \geq \frac{t_{k+1}}{\sigma_{k+1}}$,
		and $t_{k}= t_{k+1}\lambda_{k+1}= t_{k+1}\theta_{k+1} $. Then
\[ 
			t_{k} b_{k+1}(x,y) \geq t_{k+1} a_{k+1} (x,y).
\]
	Consequently, if $(\forall k \in \mathbf{N})$ 
	$\lambda_{k}=\theta_{k} \equiv 1$,  
	$\tau_{k+1} \geq \frac{\tau_{k}}{1+\mu_{1}\tau_{k}}$, and $\sigma_{k+1} \geq 
	\frac{\sigma_{k}}{1+\mu_{2}\sigma_{k}}$. Then $(\forall (x,y) \in 
	\mathcal{H}_{1}\times \mathcal{H}_{2})$ $b_{k+1}(x,y) \geq a_{k+1} (x,y)$.
	
		\item \label{lemma:abck:ck} Suppose that there exist $\eta_{x}$ and $\eta_{y}$ in 
		$\mathbf{R}_{++}$ such that 
		
		\begin{align*}
			&\left(\frac{1}{\tau_{k}}- \lambda_{k} \alpha_{k}L_{xx} 
			-
			\lambda_{k}\beta_{k}L_{xy}- \frac{ L_{xx}}{ \alpha_{k+1}} -
			\frac{ L_{yx}}{ \gamma_{k+1}}  \right)  \geq \frac{\eta_{x}}{\tau_{k}}, \text{ and}\\
			&\left(\frac{1}{\sigma_{k}} -\theta_{k}\gamma_{k}L_{yx} 
			-\theta_{k}\delta_{k}L_{yy}
			-\frac{L_{xy}}{\beta_{k+1}}-\frac{L_{yy}}{\delta_{k+1}}  
			\right)  \geq \frac{\eta_{y}}{\sigma_{k}}.
		\end{align*}
Then  
\[
			c_{k} \geq \frac{1}{2} \left( 
			\frac{\eta_{x}}{\tau_{k}} \norm{x^{k+1}-x^{k}}^{2} +
			\frac{\eta_{y}}{\sigma_{k}} \norm{y^{k+1}-y^{k}}^{2}  \right) \geq 0.
 \]
	\end{enumerate}
\end{lemma}

\begin{proof}

	\cref{lemma:abck:bk}: According to \cref{lemma:fxk+1yxyk+1:ak},
	\cref{lemma:fxk+1yxyk+1:bk}, and our assumptions, 
	\begin{align*}
		&t_{k} b_{k+1}(x,y) \\
		\stackrel{\cref{lemma:fxk+1yxyk+1:bk}}{=}&  \frac{t_{k} }{2} 
		\left(\frac{1}{\tau_{k}}+\mu_{1}\right)
		\norm{x-x^{k+1}}^{2}  
		+\frac{t_{k} }{2} \left(\frac{1}{\sigma_{k}} +\mu_{2} \right)\norm{y-y^{k+1}}^{2} 
		\\
		&+t_{k} \left( \frac{ L_{xx}}{2\alpha_{k+1}} + 
		\frac{ L_{yx}}{2\gamma_{k+1}} \right) \norm{x^{k+1}-x^{k}}^{2} 
		+t_{k} \left(\frac{L_{xy}}{2\beta_{k+1}}+\frac{L_{yy}}{2\delta_{k+1}}  \right) 
		\norm{y^{k+1}-y^{k}}^{2}\\
		& +\scalemath{0.9}{t_{k} \innp{\nabla_{x}\Phi 	(x^{k+1},y^{k+1})-\nabla_{x} 
		\Phi(x^{k},y^{k}), 
		x-x^{k+1}} 
		-t_{k} \innp{\nabla_{y}\Phi(x^{k+1},y^{k+1})-\nabla_{y} \Phi(x^{k},y^{k}), y-y^{k+1} 
		} } \\
		\geq&
		\frac{t_{k+1}}{2\tau_{k+1}} \norm{x-x^{k+1}}^{2} 
		+\frac{t_{k+1}}{2\sigma_{k+1}}\norm{y-y^{k+1}} \\
		&+t_{k+1}\left( \frac{\lambda_{k+1}L_{xx}}{2\alpha_{k+1}} + 
		\frac{\theta_{k+1}L_{yx}}{2\gamma_{k+1}} \right) \norm{x^{k+1}-x^{k}}^{2} 
		+t_{k+1}
		\left(
		\frac{\lambda_{k+1}L_{xy}}{2\beta_{k+1}}+\frac{\theta_{k+1}L_{yy}}{2\delta_{k+1}} 
		\right) 
		\norm{y^{k+1}-y^{k}}^{2}\\
		&+t_{k+1}\lambda_{k+1} \innp{\nabla_{x}\Phi(x^{k+1},y^{k+1}) 
			-\nabla_{x}\Phi(x^{k},y^{k}),x-x^{k+1}} \\
		&  -t_{k+1}\theta_{k+1} \innp{	 
			\nabla_{y}\Phi(x^{k+1},y^{k+1}) 
			-\nabla_{y}\Phi(x^{k},y^{k}),y-y^{k+1}  }\\
		\stackrel{\cref{lemma:fxk+1yxyk+1:ak}}{=}&t_{k+1}a_{k+1}(x,y).
	\end{align*}
	
	\cref{lemma:abck:ck}: In view of our assumptions and \cref{lemma:fxk+1yxyk+1:ck}, 
	\begin{align*}
		c_{k} = &\frac{1}{2}\left(\frac{1}{\tau_{k}}- \lambda_{k} \alpha_{k}L_{xx} -
		\lambda_{k}\beta_{k}L_{xy}- \frac{ L_{xx}}{ \alpha_{k+1}} -
		\frac{ L_{yx}}{ \gamma_{k+1}}  \right) \norm{x^{k+1}-x^{k}}^{2} \\
		&+\frac{1}{2 } \left(\frac{1}{\sigma_{k}} -\theta_{k}\gamma_{k}L_{yx} 
		-\theta_{k}\delta_{k}L_{yy}
		-\frac{L_{xy}}{\beta_{k+1}}-\frac{L_{yy}}{\delta_{k+1}}  
		\right) 
		\norm{y^{k+1}-y^{k}}^{2}\\
		\geq &\frac{1}{2} \left( \frac{\eta_{x}}{\tau_{k}} \norm{x^{k+1}-x^{k}}^{2} +
		\frac{\eta_{y}}{\sigma_{k}} \norm{y^{k+1}-y^{k}}^{2}  \right) \geq 0.
	\end{align*}
\end{proof}

\begin{proposition}\label{prop:abckparameters}
\begin{enumerate}
	\item \label{prop:abckparameters:ab} Suppose that $(\forall k \in \mathbf{N})$ 
	$\lambda_{k} =\theta_{k}$ and $t_{k}=\frac{t_{0}}{\theta_{1}\cdots \theta_{k}}$ with
	$\lambda_{0} =\theta_{0}=1$ and $t_{0} \in \mathbf{R}_{++}$.
	Suppose that $(\forall k \in \mathbf{N})$ $\tau_{k+1} \geq \frac{\tau_{k}}{\theta_{k+1} 
		(1+\mu_{1}\tau_{k})}$ and $\sigma_{k+1} \geq \frac{\sigma_{k}}{\theta_{k+1} 
		(1+\mu_{2}\sigma_{k})}$. Then $(\forall k \in \mathbf{N})$ $t_{k} b_{k+1}(x,y) \geq 
		t_{k+1} a_{k+1} (x,y)$.
	\item \label{prop:abckparameters:c} 
Suppose  $\eta_{x}:=1-\sup_{k \in \mathbf{N}}\tau_{k}\left(   
L_{xx}\sup_{k \in \mathbf{N}}\lambda_{k}\alpha_{k} 
+L_{xy}\sup_{k \in \mathbf{N}}\lambda_{k}\beta_{k} +\frac{L_{xx}}{\inf_{k \in 
		\mathbf{N}}\alpha_{k}} +\frac{L_{yx}}{\inf_{k \in \mathbf{N}}\gamma_{k}}  
\right)$ $ \in \mathbf{R}_{++}$
and 
$\eta_{y}:= 1- \sup_{k \in \mathbf{N}}\sigma_{k} \left(  L_{yx}\sup_{k \in \mathbf{N}} 
\theta_{k} \gamma_{k} + 
L_{yy}\sup_{k \in \mathbf{N}}\theta_{k}\delta_{k} + \frac{L_{xx}}{\inf_{k \in 
		\mathbf{N}}\beta_{k}} +\frac{L_{yy}}{\inf_{k \in \mathbf{N}}\delta_{k}}  
\right) \in \mathbf{R}_{++}$.
Then $(\forall k \in \mathbf{N})$  $c_{k} \geq \frac{1}{2} \left( 
	\frac{\eta_{x}}{\tau_{k}} \norm{x^{k+1}-x^{k}}^{2} +
	\frac{\eta_{y}}{\sigma_{k}} \norm{y^{k+1}-y^{k}}^{2}  \right) \geq 0$.
\end{enumerate}
\end{proposition}

\begin{proof}
	\cref{prop:abckparameters:ab}: Clearly, our assumptions  $(\forall k \in \mathbf{N})$ 
	$\lambda_{k} =\theta_{k}$ and $t_{k}=\frac{t_{0}}{\theta_{1}\cdots \theta_{k}}$ with
	$\lambda_{0} =\theta_{0}=1$ and $t_{0} \in \mathbf{R}_{++}$ imply that $t_{k}= 
	t_{k+1}\lambda_{k+1}= t_{k+1}\theta_{k+1} $.
	Moreover, it is easy to see that for every $k \in \mathbf{N}$,
	\begin{align*}
		&\tau_{k+1} \geq \frac{\tau_{k}}{\theta_{k+1} (1+\mu_{1}\tau_{k})}  \Leftrightarrow 
		\tau_{k+1} \geq \frac{\tau_{k}t_{k+1}}{t_{k} (1+\mu_{1}\tau_{k})} \Leftrightarrow 
		t_{k}\left(\frac{1}{\tau_{k}}+\mu_{1}\right)
		\geq  \frac{t_{k+1}}{\tau_{k+1}};\\
		& \sigma_{k+1} \geq \frac{\sigma_{k}}{\theta_{k+1} (1+\mu_{2}\sigma_{k})}
		\Leftrightarrow 
		\sigma_{k+1} \geq \frac{\sigma_{k}t_{k+1}}{t_{k} (1+\mu_{2}\sigma_{k})}
		\Leftrightarrow t_{k}\left(\frac{1}{\sigma_{k}} +\mu_{2} \right) \geq 
		\frac{t_{k+1}}{\sigma_{k+1}}.
	\end{align*}
Hence, the required result follows immediately from 
\cref{lemma:abck}\cref{lemma:abck:bk}.

	\cref{prop:abckparameters:c}: According to our assumptions, we have that for every 
	$ k \in \mathbf{N}$,
		\begin{align*}
			1-\eta_{x} &=\sup_{k \in \mathbf{N}}\tau_{k}\left(   
			L_{xx}\sup_{k \in \mathbf{N}}\lambda_{k}\alpha_{k} 
			+L_{xy}\sup_{k \in \mathbf{N}}\lambda_{k}\beta_{k} +\frac{L_{xx}}{\inf_{k \in 
					\mathbf{N}}\alpha_{k}} +\frac{L_{yx}}{\inf_{k \in \mathbf{N}}\gamma_{k}}  
			\right) \\
		  & \geq \tau_{k} \left( 
		  \lambda_{k} \alpha_{k}L_{xx} 
		+  \lambda_{k}\beta_{k}L_{xy}+  \frac{ L_{xx}}{ \alpha_{k+1}} +
		  \frac{ L_{yx}}{ \gamma_{k+1}}  \right),
		\end{align*}
	and 
	\begin{align*}
1-	\eta_{y} &=  \sup_{k \in \mathbf{N}}\sigma_{k} \left(  L_{yx}\sup_{k \in \mathbf{N}} 
	\theta_{k} \gamma_{k} + 
	L_{yy}\sup_{k \in \mathbf{N}}\theta_{k}\delta_{k} + \frac{L_{xx}}{\inf_{k \in 
			\mathbf{N}}\beta_{k}} +\frac{L_{yy}}{\inf_{k \in \mathbf{N}}\delta_{k}}  
	\right) \\
	&\geq \sigma_{k} \left( \theta_{k}\gamma_{k}L_{yx} 
	+\theta_{k}\delta_{k}L_{yy}
	+\frac{L_{xy}}{\beta_{k+1}}+\frac{L_{yy}}{\delta_{k+1}}  
	\right),
	\end{align*}
which imply,  respectively, that for every 
$ k \in \mathbf{N}$,
\[ 
	\left(\frac{1}{\tau_{k}}- 
	\lambda_{k} \alpha_{k}L_{xx} 
	-
	\lambda_{k}\beta_{k}L_{xy}- \frac{ L_{xx}}{ \alpha_{k+1}} -
	\frac{ L_{yx}}{ \gamma_{k+1}}  \right) \geq \frac{\eta_{x}}{\tau_{k}},
\]
and 
\[ 
	\left(\frac{1}{\sigma_{k}} -\theta_{k}\gamma_{k}L_{yx} 
	-\theta_{k}\delta_{k}L_{yy}
	-\frac{L_{xy}}{\beta_{k+1}}-\frac{L_{yy}}{\delta_{k+1}}  
	\right)  \geq \frac{\eta_{y}}{\sigma_{k}}.
\]
	Altogether, we see that the desired results are from 
	\cref{lemma:abck}\cref{lemma:abck:ck}.
\end{proof}

\begin{remark}  \label{remark:Parameters}
	Because assumptions in \cref{lemma:abck} are critical in our convergence results 
	later,	we further discuss them below. 
	\begin{enumerate}
		\item  \label{remark:Parameters:Axy}
		Let $A: \mathcal{H}_{1} \to \mathcal{H}_{2}$ be bounded and linear. 
		Consider $(\forall (x,y)\in \mathcal{H}_{1}\times \mathcal{H}_{2})$ $\Phi(x,y) 
		=\innp{Ax, y}$. (In this case, the function $f$ defined in \cref{eq:def:f} is or covers 
		convex-concave functions studied  in a number of papers on solving 
		convex-concave saddle-point problems such as \cite{BonettiniRuggiero2012}, 
	\cite{ChambollePock2011}, \cite{HeYouYuan2014}, \cite{HamedaniAybat2021}, 
	\cite{ZhuChan2008efficient}, and so on.)
		Because, via \cite[Fact~2.25(ii)]{BauschkCombettes2017}, 
		$\norm{A} 
		=\norm{A^{*}}$, we have $L_{xx}=L_{yy} =0$ and 
		$L_{xy}=L_{yx} =\norm{A}$ in our assumption \cref{eq:nablaPhixy}. If $(\forall k \in 
		\mathbf{N})$ $\lambda_{k}=\theta_{k}=1$, $\tau_{k}=\tau \in \mathbf{R}_{++}$, 
		$\sigma_{k} =\sigma  \in \mathbf{R}_{++}$,  $\beta_{k}=\beta  \in \mathbf{R}_{++}$, 
		and  $\gamma_{k} 
		=\gamma  \in \mathbf{R}_{++}$, then
		the assumptions in \cref{prop:abckparameters}\cref{prop:abckparameters:ab} are 
		satisfied automatically, and 	the assumption in
		\cref{prop:abckparameters}\cref{prop:abckparameters:c}  becomes
\[ 
			\eta_{x} := 1- \tau \left(\beta +\frac{1}{\gamma}\right)\norm{A}  >0  \quad 
			\text{and} 
			\quad 
			\eta_{y} :=1- 	\sigma \left( \gamma +\frac{1}{\beta}\right)\norm{A} >0.
\]
	
		\item \label{remark:Parameters:bounded}
		One easy example satisfying assumptions of 
		\cref{prop:abckparameters}\cref{prop:abckparameters:ab} is: 
		$(\forall k \in \mathbf{N})$ $\lambda_{k} =\theta_{k} \equiv 1$, $t_{k} \equiv t_{0} 
		\in \mathbf{R}_{++}$, $\tau_{k}  \equiv  \tau_{0} \in \mathbf{R}_{++}$, and
		$\sigma_{k}  \equiv  \sigma_{0} \in \mathbf{R}_{++}$.
		
		In addition,  if $\sup_{k \in \mathbf{N}} \lambda_{k} <\infty$, 
		$\sup_{k \in \mathbf{N}} \theta_{k} <\infty$,
		$0 < \inf_{k \in \mathbf{N}} \alpha_{k} \leq \sup_{k \in \mathbf{N}} \alpha_{k} 
		<\infty$,  
		$0 < \inf_{k \in \mathbf{N}} \beta_{k} \leq \sup_{k \in \mathbf{N}} \beta_{k} <\infty$, 
		$0 < \inf_{k \in \mathbf{N}} \gamma_{k} \leq \sup_{k \in \mathbf{N}} \gamma_{k}
		<\infty$,
		and  $0 < \inf_{k \in \mathbf{N}} \delta_{k} \leq \sup_{k \in \mathbf{N}} \delta_{k} 
		<\infty$,
		then we can always find $(\tau_{k})_{k \in \mathbf{N}}$ and $(\sigma_{k})_{k \in 
			\mathbf{N}}$ small enough to satisfy assumptions stated in 
		\cref{prop:abckparameters}\cref{prop:abckparameters:c}.    
		
		Note that the sequences $(\tau_{k})_{k \in \mathbf{R}}$ and $(\sigma_{k})_{k \in 
		\mathbf{R}}$
		are the parameters used in our iteration scheme \cref{eq:algorithmproxif1f2}. If it is 
		necessary, we can always set them as small as possible. 
	\end{enumerate}
	
\end{remark}

\begin{lemma} \label{lemma:fergodic}
	Let $(x,y)$ be in $\mathcal{H}_{1} \times \mathcal{H}_{2}$,
	let $K$ be in $ \mathbf{N} \smallsetminus \{0\}$, 
	and let $(\forall k \in \mathbf{N})$
	$a_{k}(x,y)$, $b_{k+1}(x,y)$, and $c_{k}$ be defined in 
	\cref{lemma:fxk+1yxyk+1:ak}, 
	\cref{lemma:fxk+1yxyk+1:bk}, and \cref{lemma:fxk+1yxyk+1:ck}, respectively. 
	We have the following assertions. 
	\begin{enumerate}
		\item \label{lemma:fergodic:ineq} Let $(x,y) \in \dom f_{1} \times \dom 
		f_{2}$. Then
		\begin{align*}
			f(\hat{x}_{K},y) -f(x,\hat{y}_{K}) \leq&
			\frac{1}{\sum^{K-1}_{i=0}t_{i}}  
			\sum^{K-1}_{k=0}t_{k} \left( f(x^{k+1},y) -f(x,y^{k+1})\right)\\
			\leq &\frac{1}{\sum^{K-1}_{i=0}t_{i}}  
			\sum^{K-1}_{k=0}t_{k} \left( a_{k}(x,y) -b_{k+1}(x,y) -c_{k}\right).
		\end{align*}
		\item  \label{lemma:fergodic:abk} 
		Suppose that $(\forall k \in \mathbf{N})$ 
		$\lambda_{k} =\theta_{k}$ and $t_{k}=\frac{t_{0}}{\theta_{1}\cdots \theta_{k}}$ with
		$\lambda_{0} =\theta_{0}=1$ and $t_{0} \in \mathbf{R}_{++}$.
		Suppose that $(\forall k \in \mathbf{N})$ $\tau_{k+1} \geq 
		\frac{\tau_{k}}{\theta_{k+1} 
			(1+\mu_{1}\tau_{k})}$ and $\sigma_{k+1} \geq \frac{\sigma_{k}}{\theta_{k+1} 
			(1+\mu_{2}\sigma_{k})}$. 
		Suppose that  $\eta_{x}:=1-\sup_{k \in \mathbf{N}}\tau_{k}\left(   
		L_{xx}\sup_{k \in \mathbf{N}}\lambda_{k}\alpha_{k} 
		+L_{xy}\sup_{k \in \mathbf{N}}\lambda_{k}\beta_{k} +\frac{L_{xx}}{\inf_{k \in 
				\mathbf{N}}\alpha_{k}} +\frac{L_{yx}}{\inf_{k \in \mathbf{N}}\gamma_{k}}  
		\right) \in \mathbf{R}_{++}$
		and 
		$\eta_{y}:= 1- \sup_{k \in \mathbf{N}}\sigma_{k} \left(  L_{yx}\sup_{k \in \mathbf{N}} 
		\theta_{k} \gamma_{k} + 
		L_{yy}\sup_{k \in \mathbf{N}}\theta_{k}\delta_{k} + \frac{L_{xx}}{\inf_{k \in 
				\mathbf{N}}\beta_{k}} +\frac{L_{yy}}{\inf_{k \in \mathbf{N}}\delta_{k}}  
		\right) \in \mathbf{R}_{++}$.
		Then we have the following statements. 
		\begin{enumerate}
			\item  \label{lemma:fergodic:abk:f}  For every $K \in \mathbf{N} \smallsetminus 
			\{0\}$,
			\begin{align*}
				&\frac{1}{\sum^{K-1}_{i=0}t_{i}}  
				\sum^{K-1}_{k=0}t_{k} \left( a_{k}(x,y) -b_{k+1}(x,y) -c_{k}\right)\\
				\leq & \frac{t_{0}}{2\sum^{K-1}_{i=0}t_{i}}   
				\left( \frac{1}{\tau_{0}} 
				\norm{x-x^{0}}^{2} 
				+\frac{1}{\sigma_{0}}\norm{y-y^{0}}^{2} \right)\\
				&-  
				\frac{t_{K}}{2\sum^{K-1}_{i=0}t_{i}}
				\left(\frac{1}{\tau_{K}} -\lambda_{K} \left( L_{xx}\alpha_{K} +
				L_{xy}\beta_{K}\right) 
				\right)
				\norm{x -x^{K}}^{2}\\
				&-  \frac{t_{K}}{2\sum^{K-1}_{i=0}t_{i}} 
				\left(  \frac{1}{\sigma_{K}} -\theta_{K} 
				\left(L_{yx}\gamma_{K}+L_{yy}\delta_{K}  
				\right)\right)
				\norm{y -y^{K}}^{2}.
			\end{align*}

			\item \label{lemma:fergodic:abk:0leqf}  Let $(x^{*},y^{*})$ be a saddle-point of 
			$f$. Then for every $K \in \mathbf{N} \smallsetminus \{0\}$, we have that 
			\begin{align*}
				0 \leq &f(\hat{x}_{K},y^{*}) -f(x^{*},\hat{y}_{K})\\
				\leq  &\frac{1}{\sum^{K-1}_{i=0}t_{i}}  
				\sum^{K-1}_{k=0}t_{k} \left( a_{k}(x^{*},y^{*}) -b_{k+1}(x^{*},y^{*}) 
				-c_{k}\right)\\
				\leq &\frac{t_{0}}{2\sum^{K-1}_{i=0}t_{i}}   
				\left( \frac{1}{\tau_{0}} 
				\norm{x^{*}-x^{0}}^{2} 
				+\frac{1}{\sigma_{0}}\norm{y^{*}-y^{0}}^{2} \right)\\
				&-  \frac{t_{K}}{2\sum^{K-1}_{i=0}t_{i}}
				\left(\frac{1}{\tau_{K}} -\lambda_{K} \left( L_{xx}\alpha_{K} +
				L_{xy}\beta_{K}\right) 
				\right)
				\norm{x^{*} -x^{K}}^{2}\\
				&-  \frac{t_{K}}{2\sum^{K-1}_{i=0}t_{i}} 
				\left(  \frac{1}{\sigma_{K}} -\theta_{K} 
				\left(L_{yx}\gamma_{K}+L_{yy}\delta_{K}  
				\right)\right)
				\norm{y^{*} -y^{K}}^{2}.
			\end{align*}
		\end{enumerate}
		
	\end{enumerate}
\end{lemma}

\begin{proof}
	\cref{lemma:fergodic:ineq}: 	Because $f(\cdot, y)-f(x,\cdot)$ is convex, we know 
	that
	\begin{align*}
		f(\hat{x}_{K},y) -f(x,\hat{y}_{K}) &\leq  
		\frac{1}{\sum^{K-1}_{i=0}t_{i}}  
		\sum^{K-1}_{k=0}t_{k} \left( f(x^{k+1},y) -f(x,y^{k+1})\right)\\
		&\leq  \frac{1}{\sum^{K-1}_{i=0}t_{i}}  
		\sum^{K-1}_{k=0}t_{k} \left( a_{k}(x,y) -b_{k+1}(x,y) -c_{k}\right),
	\end{align*}
	where in the second inequality, we used \cref{prop:fxk+1yxyk+1}.

	\cref{lemma:fergodic:abk:f}: 
Let $K $ be in $ \mathbf{N} \smallsetminus \{0\}$.
	Adopting 	\cref{lemma:abck} in the first	inequality,
 employing   $x^{-1}=x^{0}$ and $y^{-1}=y^{0}$ in the second equality, 
 and using	
 \cref{lemma:proxf1f2preliminary}\cref{lemma:proxf1f2preliminary:Lipschitzx}$\&$\cref{lemma:proxf1f2preliminary:Lipschitzy}
	with   $(\bar{x},\bar{y}) =(x^{K},y^{K})$, $(\hat{x},\hat{y}) = (x^{K-1},y^{K-1})$, 
	$(\alpha,\beta) = (\alpha_{K},\beta_{K})$,  and $(\gamma, \delta) = 
	(\gamma_{K},\delta_{K})$ in the second inequality, we have that
	\begin{align*}
		&\frac{1}{\sum^{K-1}_{i=0}t_{i}}  
		\sum^{K-1}_{k=0}t_{k} \left( a_{k}(x,y) -b_{k+1}(x,y) 
		-c_{k}\right)\\
		\leq & \frac{1}{\sum^{K-1}_{i=0}t_{i}}  
		\sum^{K-1}_{k=0}t_{k}  a_{k}(x,y) - t_{k+1}a_{k+1}(x,y)\\
		= &  \frac{t_{0}}{\sum^{K-1}_{i=0}t_{i}}  a_{0}(x,y)  -  
		\frac{t_{K}}{\sum^{K-1}_{i=0}t_{i}}   a_{K}(x,y) \\
		\stackrel{\cref{lemma:fxk+1yxyk+1:ak}}{=}&   \frac{t_{0}}{2\sum^{K-1}_{i=0}t_{i}}   
		\left( \frac{1}{\tau_{0}} 
		\norm{x-x^{0}}^{2} 
		+\frac{1}{\sigma_{0}}\norm{y-y^{0}}^{2}  \right)
		-  \frac{t_{K}}{2\sum^{K-1}_{i=0}t_{i}} \left(  \frac{1}{\tau_{K}} 
		\norm{x-x^{K}}^{2} 
		+\frac{1}{\sigma_{K}}\norm{y-y^{K}}^{2}  \right) \\
		& \scalemath{0.9}{- \frac{t_{K}}{2\sum^{K-1}_{i=0}t_{i}}\left( 
		\frac{L_{xx}\lambda_{K}}{\alpha_{K}} + 
		\frac{ L_{yx}\theta_{K}}{\gamma_{K}} \right) \norm{x^{K}-x^{K-1}}^{2} 
		- 
		\frac{t_{K}}{2\sum^{K-1}_{i=0}t_{i}}\left(\frac{L_{xy}\lambda_{K}}{\beta_{K}}+\frac{ 
			L_{yy}\theta_{K}}{\delta_{K}} 
		\right) 
		\norm{y^{K}-y^{K-1}}^{2} }\\
		&-  \frac{\lambda_{K} t_{K}}{\sum^{K-1}_{i=0}t_{i}}  
		\innp{\nabla_{x}\Phi(x^{K},y^{K}) 
			-\nabla_{x}\Phi(x^{K-1},y^{K-1}),x-x^{K}}  \\ 
		&+ \frac{\theta_{K}  t_{K}}{\sum^{K-1}_{i=0}t_{i}}   
		\innp{	 
			\nabla_{y}\Phi(x^{K},y^{K}) 
			-\nabla_{y}\Phi(x^{K-1},y^{K-1}),y-y^{K}  } \\
		\leq &\frac{t_{0}}{2\sum^{K-1}_{i=0}t_{i}}   
		\left( \frac{1}{\tau_{0}} 
		\norm{x-x^{0}}^{2} 
		+\frac{1}{\sigma_{0}}\norm{y-y^{0}}^{2}  \right)
		-  \frac{t_{K}}{2\sum^{K-1}_{i=0}t_{i}} \left(  \frac{1}{\tau_{K}} 
		\norm{x-x^{K}}^{2} 
		+\frac{1}{\sigma_{K}}\norm{y-y^{K}}^{2} \right) \\
		&- \scalemath{0.9}{ \frac{t_{K}}{2\sum^{K-1}_{i=0}t_{i}}\left( 
		\frac{L_{xx}\lambda_{K}}{\alpha_{K}} + 
		\frac{ L_{yx}\theta_{K}}{\gamma_{K}} \right) \norm{x^{K}-x^{K-1}}^{2} 
		- 
		\frac{t_{K}}{2\sum^{K-1}_{i=0}t_{i}}\left(\frac{L_{xy}\lambda_{K}}{\beta_{K}}+\frac{ 
			L_{yy}\theta_{K}}{\delta_{K}} 
		\right) 
		\norm{y^{K}-y^{K-1}}^{2} }\\
		&+\scalemath{0.85}{ \frac{L_{xx}t_{K}\lambda_{K}}{2\sum^{K-1}_{i=0}t_{i}} \left( 
		\alpha_{K} 
		\norm{x -x^{K}}^{2} +\frac{1}{\alpha_{K} } \norm{x^{K}
			-x^{K-1}}^{2}   \right) + \frac{L_{xy}t_{K}\lambda_{K}}{2\sum^{K-1}_{i=0}t_{i}}  
		\left( \beta_{K}  \norm{x -x^{K}}^{2} 
		+\frac{1}{\beta_{K} } 
		\norm{y^{K} -y^{K-1}}^{2}  \right)} \\
		&+\scalemath{0.85}{ \frac{L_{yx}t_{K}\theta_{K}}{2\sum^{K-1}_{i=0}t_{i}}  \left( 
		\gamma_{K} 
		\norm{y -y^{K}}^{2} +\frac{1}{\gamma_{K} } \norm{x^{K}
			-x^{K-1}}^{2}   \right) + \frac{L_{yy}t_{K}\theta_{K}}{2\sum^{K-1}_{i=0}t_{i}}  
			\left( 
		\delta_{K}  \norm{y -y^{K}}^{2} 
		+\frac{1}{\delta_{K} } 
		\norm{y^{K} -y^{K-1}}^{2}  \right)}\\
		=& \frac{t_{0}}{2\sum^{K-1}_{i=0}t_{i}}   
		\left( \frac{1}{\tau_{0}} 
		\norm{x-x^{0}}^{2} 
		+\frac{1}{\sigma_{0}}\norm{y-y^{0}}^{2}  \right)\\
		&-  \frac{t_{K}}{2\sum^{K-1}_{i=0}t_{i}}
		\left(\frac{1}{\tau_{K}} -\lambda_{K} \left( L_{xx}\alpha_{K} + L_{xy}\beta_{K}\right) 
		\right)
		\norm{x -x^{K}}^{2}\\
		&-  \frac{t_{K}}{2\sum^{K-1}_{i=0}t_{i}} 
		\left(  \frac{1}{\sigma_{K}} -\theta_{K} \left(L_{yx}\gamma_{K}+L_{yy}\delta_{K}  
		\right)\right)
		\norm{y -y^{K}}^{2}.
	\end{align*}

	\cref{lemma:fergodic:abk:0leqf}:  Because $(x^{*},y^{*})$ is a saddle-point of $f$, 
	due to \cref{eq:saddle-point}, it 	is 	easy to see  that  
	$0 \leq		f(\hat{x}_{K},y^{*}) 	-f(x^{*},\hat{y}_{K})$.
	Therefore, the required result follows immediately from \cref{lemma:fergodic:ineq} 
	and \cref{lemma:fergodic:abk:f}	above. 
\end{proof}

\subsection{Convergence under Convex Assumptions} 
\label{subsection:convergence}

Recall from our global assumption \cref{assumption} that 
$(\forall i \in \{1,2\})$ $f_{i}$ is convex with modulus $\mu_{i} \geq 0$.
In this section, we require neither $\mu_{1} >0$ nor $\mu_{2} >0$, that is, 
$f_{1}$ and $f_{2}$ are only required to be convex but not strongly convex.

\begin{theorem} \label{theorem:fconverge}
	Let $(x,y)$ be in $\mathcal{H}_{1} \times \mathcal{H}_{2}$, let $(x^{*},y^{*})$ 
	be a saddle-point of $f$, and let 
	$(\forall k \in \mathbf{N})$ $a_{k}(x,y)$, $b_{k+1}(x,y)$, and $c_{k}$ be defined in 
	\cref{lemma:fxk+1yxyk+1:ak}, 
	\cref{lemma:fxk+1yxyk+1:bk}, and \cref{lemma:fxk+1yxyk+1:ck}, respectively. 
	Suppose that $(\forall k \in \mathbf{N})$ 
	$\lambda_{k} =\theta_{k} \leq 1$ and $t_{k}=\frac{t_{0}}{\theta_{1}\cdots \theta_{k}}$ 
	with
	$\lambda_{0} =\theta_{0}=1$ and $t_{0} \in \mathbf{R}_{++}$, and
  that $(\forall k \in \mathbf{N})$ $\tau_{k+1} \geq 
	\frac{\tau_{k}}{\theta_{k+1} 
		(1+\mu_{1}\tau_{k})}$ and $\sigma_{k+1} \geq \frac{\sigma_{k}}{\theta_{k+1} 
		(1+\mu_{2}\sigma_{k})}$. 
	
	Suppose that  $\eta_{x}:=1-\sup_{k \in \mathbf{N}}\tau_{k}\left(   
	L_{xx}\sup_{k \in \mathbf{N}}\lambda_{k}\alpha_{k} 
	+L_{xy}\sup_{k \in \mathbf{N}}\lambda_{k}\beta_{k} +\frac{L_{xx}}{\inf_{k \in 
			\mathbf{N}}\alpha_{k}} +\frac{L_{yx}}{\inf_{k \in \mathbf{N}}\gamma_{k}}  
	\right) \in \mathbf{R}_{++}$
	and 
	$\eta_{y}:= 1- \sup_{k \in \mathbf{N}}\sigma_{k} \left(  L_{yx}\sup_{k \in \mathbf{N}} 
	\theta_{k} \gamma_{k} + 
	L_{yy}\sup_{k \in \mathbf{N}}\theta_{k}\delta_{k} + \frac{L_{xx}}{\inf_{k \in 
			\mathbf{N}}\beta_{k}} +\frac{L_{yy}}{\inf_{k \in \mathbf{N}}\delta_{k}}  
	\right) \in \mathbf{R}_{++}$.
	We have the following statements.  
		\begin{enumerate}
			\item \label{theorem:fconverge:minmaxgap} 
			$0 \leq f(\hat{x}_{K},y^{*}) -f(x^{*},\hat{y}_{K}) \leq \frac{1}{K}   
			\left( \frac{1}{2\tau_{0}} 
			\norm{x^{*}-x^{0}}^{2} 
			+\frac{1}{2\sigma_{0}}\norm{y^{*}-y^{0}} \right)$.
			\item \label{theorem:fconverge:tkak} $(\forall k \in \mathbf{N})$ 
			$t_{k}a_{k}(x^{*},y^{*}) \geq 
			t_{k+1}a_{k+1}(x^{*},y^{*}) +\frac{t_{0}}{2} \left( 
			\frac{\eta_{x}}{\tau_{k}} \norm{x^{k+1}-x^{k}}^{2} +
			\frac{\eta_{y}}{\sigma_{k}} \norm{y^{k+1}-y^{k}}^{2}  
			\right)$.
			\item   \label{theorem:fconverge:basic} $\lim_{k \to \infty} 
			t_{k}a_{k}(x^{*},y^{*}) \in \mathbf{R}_{+}$ exists and $(x^{k})_{k \in 
			\mathbf{N}}$ and 
			$(y^{k})_{k \in \mathbf{N}}$ are bounded. 
			
				\item  \label{theorem:fconverge:f}  $f(\hat{x}_{k},\hat{y}_{k}) 
			-f(x^{*},y^{*}) \to 0$.
		\end{enumerate}
	\end{theorem}

\begin{proof}
		Note that our assumptions $(\forall k \in \mathbf{N})$ 
	$\lambda_{k} =\theta_{k} \leq 1$ and $t_{k}=\frac{t_{0}}{\theta_{1}\cdots \theta_{k}}$ 
	with	$\lambda_{0} =\theta_{0}=1$ and $t_{0} \in \mathbf{R}_{++}$ ensure that
\begin{align}  \label{eq:prop:proxproxdiff:tk}
		(\forall k \in \mathbf{N}) \quad t_{k} \geq t_{0}>0,
\end{align}
	which guarantees that
\begin{align} \label{eq:prop:proxproxdiff:sumtk}
		(\forall K \in \mathbf{N} \smallsetminus \{0\}) \quad	\sum_{i=0}^{K-1}t_{i} \geq 
		Kt_{0}>0.
\end{align}
	Because $(x^{*},y^{*})$ is a saddle-point of $f$,
	via \cref{eq:saddle-point}, we know that 
	$f(x^{k+1},y^{*}) -f(x^{*},y^{k+1})   \geq 0$. Combine this with 
	\cref{prop:fxk+1yxyk+1} and \cref{lemma:abck}\cref{lemma:abck:bk}  to derive that 
	\begin{subequations}\label{eq:prop:proxproxdiff}
		\begin{align}
			0 &\leq t_{k} \left( f(x^{k+1},y^{*}) -f(x^{*},y^{k+1}) \right) \\
			&\leq t_{k} \left( a_{k}(x^{*},y^{*}) - b_{k+1}(x^{*},y^{*}) 
			- c_{k}  \right)\\
			&\leq t_{k}a_{k}(x^{*},y^{*}) -t_{k+1}a_{k+1}(x^{*},y^{*}) 
			-t_{k}c_{k}.
		\end{align}	
	\end{subequations}
	
	Clearly,   $\eta_{x}:=1-\sup_{k \in \mathbf{N}}\tau_{k}\left(   
	L_{xx}\sup_{k \in \mathbf{N}}\lambda_{k}\alpha_{k} 
	+L_{xy}\sup_{k \in \mathbf{N}}\lambda_{k}\beta_{k} +\frac{L_{xx}}{\inf_{k \in 
			\mathbf{N}}\alpha_{k}} +\frac{L_{yx}}{\inf_{k \in \mathbf{N}}\gamma_{k}}  
	\right) \in \mathbf{R}_{++}$
	and 
	$\eta_{y}:= 1- \sup_{k \in \mathbf{N}}\sigma_{k} \left(  L_{yx}\sup_{k \in \mathbf{N}} 
	\theta_{k} \gamma_{k} + 
	L_{yy}\sup_{k \in \mathbf{N}}\theta_{k}\delta_{k} + \frac{L_{xx}}{\inf_{k \in 
			\mathbf{N}}\beta_{k}} +\frac{L_{yy}}{\inf_{k \in \mathbf{N}}\delta_{k}}  
	\right) \in \mathbf{R}_{++}$ imply, respectively,  that 
	\begin{subequations}\label{eq:prop:proxproxdiff:exists:minmaxgap:sup}
		\begin{align}
			&L_{1}^{a}:= \frac{1}{\sup_{k \in \mathbf{N}}\tau_{k}} -L_{xx}\sup_{k \in 
				\mathbf{N}}\lambda_{k}\alpha_{k} 
			-L_{xy}\sup_{k \in \mathbf{N}}\lambda_{k}\beta_{k}  >0;\\
			&L_{2}^{a} := \frac{1}{\sup_{k \in \mathbf{N}} \sigma_{k \in \mathbf{N}}} -
			L_{yx}\sup_{k \in \mathbf{N}} 
			\theta_{k} \gamma_{k} -
			L_{yy}\sup_{k \in \mathbf{N}}\theta_{k}\delta_{k}>0.
		\end{align}
	\end{subequations}

	Combine our assumptions with \cref{lemma:akgeq} and   
	\cref{lemma:abck}\cref{lemma:abck:ck} to deduce that 
\begin{align}  \label{eq:prop:proxproxdiff:exists:ak:ak}
		(\forall k \in \mathbf{N}) \quad a_{k}(x^{*},y^{*}) \geq 
		L^{a}_{1}\norm{x^{*}-x^{k}}^{2} 
		+L^{a}_{2}\norm{y^{*}-y^{k}}^{2}\geq 0.
\end{align}

		Applying
	\cref{lemma:fergodic}\cref{lemma:fergodic:ineq}$\&$\cref{lemma:fergodic:abk:f} in 
	the first two inequalities and employing 
	\cref{eq:prop:proxproxdiff:exists:minmaxgap:sup} and 
	\cref{eq:prop:proxproxdiff:sumtk} in the last inequality, we 
	have that for every $K \in \mathbf{N} \smallsetminus \{0\}$,
	\begin{subequations}\label{eq:prop:proxproxdiff:exists}
		\begin{align}
			&\frac{1}{\sum^{K-1}_{i=0}t_{i}}  
			\sum^{K-1}_{k=0}t_{k} \left( f(x^{k+1},y) -f(x,y^{k+1})\right)\\
			\leq&\frac{1}{\sum^{K-1}_{i=0}t_{i}}  
			\sum^{K-1}_{k=0}t_{k} \left( a_{k}(x,y) -b_{k+1}(x,y) -c_{k}\right)\\
			\leq & \frac{t_{0}}{2\sum^{K-1}_{i=0}t_{i}}   
			\left( \frac{1}{\tau_{0}} 
			\norm{x-x^{0}}^{2} 
			+\frac{1}{\sigma_{0}}\norm{y-y^{0}}^{2} \right)\\
			&-  
			\frac{t_{K}}{2\sum^{K-1}_{i=0}t_{i}}
			\left(\frac{1}{\tau_{K}} -\lambda_{K} \left( L_{xx}\alpha_{K} +
			L_{xy}\beta_{K}\right) 
			\right)
			\norm{x -x^{K}}^{2}\\
			&-  \frac{t_{K}}{2\sum^{K-1}_{i=0}t_{i}} 
			\left(  \frac{1}{\sigma_{K}} -\theta_{K} 
			\left(L_{yx}\gamma_{K}+L_{yy}\delta_{K}  
			\right)\right)
			\norm{y -y^{K}}^{2}\\
			\leq & \frac{1}{2K}   
			\left( \frac{1}{\tau_{0}} 
			\norm{x-x^{0}}^{2} 
			+\frac{1}{\sigma_{0}}\norm{y-y^{0}}^{2} \right).
		\end{align}
	\end{subequations}

	\cref{theorem:fconverge:minmaxgap}:	Combine 
	 \cref{lemma:fergodic}\cref{lemma:fergodic:abk:0leqf}  with 	
	 \cref{eq:prop:proxproxdiff:exists:minmaxgap:sup} and 
	 \cref{eq:prop:proxproxdiff:sumtk} 	to derive that	
	 \begin{align*}
	 	0 \leq & f(\hat{x}_{K},y^{*}) -f(x^{*},\hat{y}_{K})\\
	 	\leq &  \frac{t_{0}}{2\sum^{K-1}_{i=0}t_{i}}   
	 	\left( \frac{1}{\tau_{0}} 
	 	\norm{x^{*}-x^{0}}^{2} 
	 	+\frac{1}{\sigma_{0}}\norm{y^{*}-y^{0}}^{2} \right)\\
	 	&-  \frac{t_{K}}{2\sum^{K-1}_{i=0}t_{i}}
	 	\left(\frac{1}{\tau_{K}} -\lambda_{K} \left( L_{xx}\alpha_{K} +
	 	L_{xy}\beta_{K}\right) 
	 	\right)
	 	\norm{x^{*} -x^{K}}^{2}\\
	 	&-  \frac{t_{K}}{2\sum^{K-1}_{i=0}t_{i}} 
	 	\left(  \frac{1}{\sigma_{K}} -\theta_{K} 
	 	\left(L_{yx}\gamma_{K}+L_{yy}\delta_{K}  
	 	\right)\right)
	 	\norm{y^{*} -y^{K}}^{2}\\
	 	\leq	&  \frac{1}{K}   
	 	\left( \frac{1}{2\tau_{0}} 
	 	\norm{x^{*}-x^{0}}^{2} 
	 	+\frac{1}{2\sigma_{0}}\norm{y^{*}-y^{0}} \right).
	 \end{align*}
	 
	 \cref{theorem:fconverge:tkak}: In view of  \cref{eq:prop:proxproxdiff}, 
	 \cref{eq:prop:proxproxdiff:exists:ak:ak},
	 \cref{lemma:abck}\cref{lemma:abck:ck}, and  \cref{eq:prop:proxproxdiff:tk}, 
	 we observe that
\[ 
	 	(\forall k \in \mathbf{N}) \quad 	t_{k}a_{k}(x^{*},y^{*}) \geq 
	 	t_{k+1}a_{k+1}(x^{*},y^{*}) +\frac{t_{0}}{2} \left( 
	 	\frac{\eta_{x}}{\tau_{k}} \norm{x^{k+1}-x^{k}}^{2} +
	 	\frac{\eta_{y}}{\sigma_{k}} \norm{y^{k+1}-y^{k}}^{2}  
	 	\right).
\]
	 
	 \cref{theorem:fconverge:basic}: According to 
	 \cref{eq:prop:proxproxdiff:exists:ak:ak} and 
	 \cref{theorem:fconverge:tkak}, we know that $\left( 
	 t_{k}a_{k}(x^{*},y^{*}) \right)_{k \in \mathbf{N}}$ is monotone nonincreasing and 
	 bounded  below by $0$, which implies that $\lim_{k \to \infty} 
	 t_{k}a_{k}(x^{*},y^{*})$ exists. 
	 
	 Furthermore, this combined with 
	 \cref{eq:prop:proxproxdiff:exists:ak:ak} and \cref{eq:prop:proxproxdiff:tk}
	 implies that $t_{0}a_{0}(x^{*},y^{*}) \geq t_{k}a_{k}(x^{*},y^{*}) \geq 
	 t_{k}\left( L^{a}_{1}\norm{x^{*}-x^{k}}^{2} 
	 +L^{a}_{2}\norm{y^{*}-y^{k}}^{2}\right) \geq t_{0}\left( 
	 L^{a}_{1}\norm{x^{*}-x^{k}}^{2} 
	 +L^{a}_{2}\norm{y^{*}-y^{k}}^{2}\right) $, which derives the boundedness of 
	 $(x^{k})_{k \in 	\mathbf{N}}$ and $(y^{k})_{k \in \mathbf{N}}$.

	 \cref{theorem:fconverge:f}: According to  \cref{lemma:fxkykx*y*},
	 \cref{lemma:fergodic}\cref{lemma:fergodic:ineq}$\&$\cref{lemma:fergodic:abk:f}, 
	 and  \cref{eq:prop:proxproxdiff:exists}, we have that for every 
	 $k \in \mathbf{N}  \smallsetminus \{0\}$,
	 \begin{subequations}\label{eq:prop:proxproxdiff:exists:f}
	 	\begin{align}
	 		&f(\hat{x}_{k},\hat{y}_{k}) -f(x^{*},y^{*}) \leq \frac{1}{k}   
	 		\left( \frac{1}{2\tau_{0}} 
	 		\norm{x^{*}-x^{0}}^{2} 
	 		+\frac{1}{2\sigma_{0}}\norm{\hat{y}_{k}-y^{0}} \right); \\
	 		&f(x^{*},y^{*}) -f(\hat{x}_{k},\hat{y}_{k})	\leq
	 		\frac{1}{k}   
	 		\left( \frac{1}{2\tau_{0}} 
	 		\norm{\hat{x}_{k}-x^{0}}^{2} 
	 		+\frac{1}{2\sigma_{0}}\norm{y^{*}-y^{0}} \right).
	 	\end{align}
	 \end{subequations}
	 Clearly, the boundedness of $((x^{k},y^{k}))_{k \in \mathbf{N}}$ obtained in 
	  \cref{theorem:fconverge:basic} yields   the boundedness of 
	 $\left(\left(\hat{x}_{k},\hat{y}_{k}\right)\right)_{k \in \mathbf{N}}$.  
	 This result together with \cref{eq:prop:proxproxdiff:exists:f} guarantees that 
\[ 
	 	\frac{1}{k}   
	 	\left( \frac{1}{2\tau_{0}} 
	 	\norm{x^{*}-x^{0}}^{2} 
	 	+\frac{1}{2\sigma_{0}}\norm{\hat{y}_{k}-y^{0}} \right) \to 0 \text{ and }
	 	\frac{1}{k}   
	 	\left( \frac{1}{2\tau_{0}} 
	 	\norm{\hat{x}_{k}-x^{0}}^{2} 
	 	+\frac{1}{2\sigma_{0}}\norm{y^{*}-y^{0}} \right) \to 0,
\]
	 as $k$ goes to infinity. Altogether, we obtain that $f(\hat{x}_{k},\hat{y}_{k}) 
	 -f(x^{*},y^{*}) \to 0$.
\end{proof}

\begin{proposition} \label{prop:proxproxdiff}
	Let $(x,y)$ be in $\mathcal{H}_{1} \times \mathcal{H}_{2}$, let $(x^{*},y^{*})$ 
	be a saddle-point of $f$, and let 
	$(\forall k \in \mathbf{N})$ $a_{k}(x,y)$, $b_{k+1}(x,y)$, and $c_{k}$ be defined in 
	\cref{lemma:fxk+1yxyk+1:ak}, 
	\cref{lemma:fxk+1yxyk+1:bk}, and \cref{lemma:fxk+1yxyk+1:ck}, respectively. 
	Suppose that $(\forall k \in \mathbf{N})$ 
	$\lambda_{k} =\theta_{k} \leq 1$ and $t_{k}=\frac{t_{0}}{\theta_{1}\cdots \theta_{k}}$ 
	with
	$\lambda_{0} =\theta_{0}=1$ and $t_{0} \in \mathbf{R}_{++}$, and
	that $(\forall k \in \mathbf{N})$ $\tau_{k+1} \geq 
	\frac{\tau_{k}}{\theta_{k+1} 
		(1+\mu_{1}\tau_{k})}$ and $\sigma_{k+1} \geq \frac{\sigma_{k}}{\theta_{k+1} 
		(1+\mu_{2}\sigma_{k})}$. 
	Suppose that $\sup_{k \in \mathbf{N}} \tau_{k} <\infty$ and $\sup_{k \in 
	\mathbf{N}}\sigma_{k}<\infty$. 
	
	Suppose that  $\eta_{x}:=1-\sup_{k \in \mathbf{N}}\tau_{k}\left(   
	L_{xx}\sup_{k \in \mathbf{N}}\lambda_{k}\alpha_{k} 
	+L_{xy}\sup_{k \in \mathbf{N}}\lambda_{k}\beta_{k} +\frac{L_{xx}}{\inf_{k \in 
			\mathbf{N}}\alpha_{k}} +\frac{L_{yx}}{\inf_{k \in \mathbf{N}}\gamma_{k}}  
	\right) \in \mathbf{R}_{++}$
	and 
	$\eta_{y}:= 1- \sup_{k \in \mathbf{N}}\sigma_{k} \left(  L_{yx}\sup_{k \in \mathbf{N}} 
	\theta_{k} \gamma_{k} + 
	L_{yy}\sup_{k \in \mathbf{N}}\theta_{k}\delta_{k} + \frac{L_{xx}}{\inf_{k \in 
			\mathbf{N}}\beta_{k}} +\frac{L_{yy}}{\inf_{k \in \mathbf{N}}\delta_{k}}  
	\right) \in \mathbf{R}_{++}$.
	
	We have the following statements.
	\begin{enumerate}
		 	\item \label{prop:proxproxdiff:limts} 
		 $\lim_{k\to \infty}x^{k+1}-x^{k}=0$ and
		 $\lim_{k \to \infty} y^{k}-y^{k+1} =0$.
		 
		 Consequently,
		 $\lim_{k \to \infty}\nabla_{x}\Phi(x^{k},y^{k}) 
		 -\nabla_{x}\Phi(x^{k-1},y^{k-1})=0$,
		 and $\lim_{k \to \infty}\nabla_{y}\Phi(x^{k},y^{k}) 
		 -\nabla_{y}\Phi(x^{k-1},y^{k-1})=0$.
		 
		 \item 	\label{prop:proxproxdiff:exists}	Suppose that $(\forall k \in \mathbf{N})$ 
		 $\lambda_{k} =\theta_{k} \equiv 1$.  Suppose that  
		 $\inf_{k \in \mathbf{N}}\alpha_{k}>0$,  $\inf_{k \in \mathbf{N}}\beta_{k}>0$, 
		 $\inf_{k \in \mathbf{N}}\gamma_{k}>0$,   and $\inf_{k \in 
		 	\mathbf{N}}\delta_{k}>0$.  
	 	Then the   following statements hold. 
		 \begin{enumerate}
		 	\item   \label{prop:proxproxdiff:exists:existence} 
		 	$\lim_{k\to \infty}\frac{1}{\tau_{k}} \norm{x^{*}-x^{k}}^{2} 
		 	+\frac{1}{\sigma_{k}}\norm{y^{*}-y^{k}}$ exists. 
		 	
		 	\item  \label{prop:proxproxdiff:exists:saddleppoint} 
		 	Suppose that $\sup_{k \in \mathbf{N}} 
		 	\tau_{k} <\infty$ and $\sup_{k \in \mathbf{N}}\sigma_{k}<\infty$.  
		 	Then every weak sequential cluster point of 
		 	$((x^{k},y^{k}))_{k \in \mathbf{N}}$  is a saddle-point of $f$.
		 \end{enumerate}
	\end{enumerate}
\end{proposition}
		
\begin{proof}
	\cref{prop:proxproxdiff:limts}: Applying   telescoping, and employing  
\cref{theorem:fconverge}\cref{theorem:fconverge:tkak} and the first result obtained in 	
\cref{theorem:fconverge}\cref{theorem:fconverge:basic} above, we have that 
\begin{align*}
	t_{0}a_{0}(x^{*},y^{*}) - \lim_{k \to \infty} 
	t_{k}a_{k}(x^{*},y^{*}) &\geq \frac{\eta_{x}t_{0}}{2} \sum_{k \in \mathbf{N}} 
	\frac{1}{\tau_{k}}
	\norm{x^{k+1}-x^{k}}^{2} 
	+ \frac{\eta_{y}t_{0}}{2} \sum_{k \in \mathbf{N}} \frac{1}{\sigma_{k}}
	\norm{y^{k+1}-y^{k}}^{2}\\
	&\geq \scalemath{0.95}{  \frac{\eta_{x}t_{0}}{2\sup_{k \in \mathbf{N}} 
		\tau_{k} } \sum_{k \in \mathbf{N}} 
	\norm{x^{k+1}-x^{k}}^{2} 
	+ \frac{\eta_{y}t_{0}}{2\sup_{k \in \mathbf{N}}\sigma_{k}} \sum_{k \in \mathbf{N}}
	\norm{y^{k+1}-y^{k}}^{2}},
\end{align*}
which, combining with our assumption $\sup_{k \in \mathbf{N}} 
\tau_{k} <\infty$ and $\sup_{k \in \mathbf{N}}\sigma_{k}<\infty$, yields that

$\lim_{k\to \infty}x^{k+1}-x^{k}=0$ and
$\lim_{k \to \infty} y^{k}-y^{k+1} =0$. 
Bearing this result in mind and applying \cref{eq:nablaPhixy} with 
$(x,y)=(x^{k},y^{k})$ and $(x',y')=(x^{k-1},y^{k-1})$, we easily deduce that 
$\lim_{k \to \infty}\nabla_{x}\Phi(x^{k},y^{k}) -\nabla_{x}\Phi(x^{k-1},y^{k-1})=0$ 
and $\lim_{k \to \infty}\nabla_{y}\Phi(x^{k},y^{k}) 
-\nabla_{y}\Phi(x^{k-1},y^{k-1})=0$.

\cref{prop:proxproxdiff:exists}: Because $(\forall k \in \mathbf{N})$ 	 $\lambda_{k} 
=\theta_{k} \equiv 1$, we know that 
\begin{align}  \label{eq:prop:proxproxdiff:exists:tk}
	(\forall k \in \mathbf{N}) \quad t_{k} \equiv t_{0} \in \mathbf{R}_{++}.
\end{align}
Hence, \cref{theorem:fconverge}\cref{theorem:fconverge:basic} guarantees that 
$\lim_{k \to \infty} 
a_{k}(x^{*},y^{*}) \in \mathbf{R}_{+}$ exists.

\cref{prop:proxproxdiff:exists:existence}: In view of \cref{lemma:fxk+1yxyk+1:ak}, 
\begin{align*}
	& a_{k}(x^{*},y^{*}) \\
	=&  \frac{1}{2\tau_{k}} \norm{x^{*}-x^{k}}^{2} 
	+\frac{1}{2\sigma_{k}}\norm{y^{*}-y^{k}}^{2} \\
	&+ \left( \frac{\lambda_{k}L_{xx}}{2\alpha_{k}} + 
	\frac{\theta_{k}L_{yx}}{2\gamma_{k}} \right) \norm{x^{k}-x^{k-1}}^{2} 
	+ \left(\frac{\lambda_{k}L_{xy}}{2\beta_{k}}+
	\frac{\theta_{k}L_{yy}}{2\delta_{k}} 	\right) 
	\norm{y^{k}-y^{k-1}}^{2}\\
	&+ \scalemath{0.95}{ \lambda_{k} \innp{\nabla_{x}\Phi(x^{k},y^{k}) 
		-\nabla_{x}\Phi(x^{k-1},y^{k-1}),x^{*}-x^{k}}   
	- \theta_{k} \innp{	 
		\nabla_{y}\Phi(x^{k},y^{k}) 
		-\nabla_{y}\Phi(x^{k-1},y^{k-1}),y^{*}-y^{k}  } }.
\end{align*}
Because $\inf_{k \in \mathbf{N}}\alpha_{k}>0$,  $\inf_{k \in \mathbf{N}}\beta_{k}>0$, 
$\inf_{k \in \mathbf{N}}\gamma_{k}>0$,  and  $\inf_{k \in \mathbf{N}}\delta_{k}>0$, 
our results obtained in \cref{theorem:fconverge}\cref{theorem:fconverge:basic} and 
\cref{prop:proxproxdiff:limts} above   ensure that 
\begin{align*}
	&\lim_{k \to \infty} \left( \frac{\lambda_{k}L_{xx}}{2\alpha_{k}} + 
	\frac{\theta_{k}L_{yx}}{2\gamma_{k}} \right) \norm{x^{k}-x^{k-1}}^{2} =0,\\
	&\lim_{k \to \infty}\left(\frac{\lambda_{k}L_{xy}}{2\beta_{k}} 
	\frac{\theta_{k}L_{yy}}{2\delta_{k}} 
	\right) 
	\norm{y^{k}-y^{k-1}}^{2}=0,\\
	&\lim_{k \to \infty}
	\lambda_{k} \innp{\nabla_{x}\Phi(x^{k},y^{k}) 
		-\nabla_{x}\Phi(x^{k-1},y^{k-1}),x^{*}-x^{k}}, \text{ and}\\
	&\lim_{k \to \infty} \theta_{k} \innp{	 
		\nabla_{y}\Phi(x^{k},y^{k}) 
		-\nabla_{y}\Phi(x^{k-1},y^{k-1}),y^{*}-y^{k}  } =0.
\end{align*} 
Results above combined with the existence of $\lim_{k \to 
	\infty}   a_{k}(x^{*},y^{*}) \in \mathbf{R}_{+}$ yield the required existence 
of $\lim_{k\to \infty}\frac{1}{\tau_{k}} \norm{x^{*}-x^{k}}^{2} 
+\frac{1}{\sigma_{k}}\norm{y^{*}-y^{k}}$.

\cref{prop:proxproxdiff:exists:saddleppoint}: Let $(\bar{x},\bar{y}) \in 
\mathcal{H}_{1}\times \mathcal{H}_{2}$ be a weakly 
sequential cluster point of $((x^{k},y^{k}))_{k \in \mathbf{N}}$, that is, there exists a 
subsequence $((x^{k_{i}},y^{k_{i}}))_{i \in \mathbf{N}}$ of $((x^{k},y^{k}))_{k \in 
	\mathbf{N}}$
such that $x^{k_{i}} \weakly \bar{x}$ and $y^{k_{i}} \weakly \bar{y}$.

Based on 
\cref{lemma:proxf1f2ineq}\cref{lemma:proxf1f2ineq:x}$\&$\cref{lemma:proxf1f2ineq:y}, 
we have that for every  $i \in \mathbf{N}$,  

$\scalemath{0.85}{\frac{x^{k_{i}} -x^{k_{i}+1}}{\tau_{k_{i}}}  - 
\left(  
\nabla_{x}\Phi(x^{k_{i}},y^{k_{i}}) 
-\nabla_{x}\Phi(x^{k_{i}-1},y^{k_{i}-1})  \right) + \left( \nabla_{x} 
\Phi(x^{k_{i}+1},y^{k_{i}+1}) 
-\nabla_{x} \Phi(x^{k_{i}},y^{k_{i}}) \right) \in  \partial_{x} 
f(x^{k_{i}+1},y^{k_{i}+1})}$,\\
$\scalemath{0.85}{\frac{y^{k_{i}} -y^{k_{i}+1}}{\sigma_{k_{i}}} 	+ 
\left(  
\nabla_{y}\Phi(x^{k_{i}},y^{k_{i}}) 
-\nabla_{y}\Phi(x^{k_{i}-1},y^{k_{i}-1})  \right) 
- \left( \nabla_{y} \Phi(x^{k_{i}+1},y^{k_{i}+1}) 
-\nabla_{y} \Phi(x^{k_{i}},y^{k_{i}}) \right) \in  \partial_{y} \left(- 
f(x^{k_{i}+1},y^{k_{i}+1}) \right)}$.  

In view of \cref{prop:proxproxdiff:limts}, we have that 

$\frac{x^{k_{i}} -x^{k_{i}+1}}{\tau_{k_{i}}}  - 
\left(  
\nabla_{x}\Phi(x^{k_{i}},y^{k_{i}}) 
-\nabla_{x}\Phi(x^{k_{i}-1},y^{k_{i}-1})  \right) + \left( \nabla_{x} 
\Phi(x^{k_{i}+1},y^{k_{i}+1}) 
-\nabla_{x} \Phi(x^{k_{i}},y^{k_{i}}) \right)  \to 0$,\\
$\frac{y^{k_{i}} -y^{k_{i}+1}}{\sigma_{k_{i}}} 	+ 
\left(  
\nabla_{y}\Phi(x^{k_{i}},y^{k_{i}}) 
-\nabla_{y}\Phi(x^{k_{i}-1},y^{k_{i}-1})  \right) 
- \left( \nabla_{y} \Phi(x^{k_{i}+1},y^{k_{i}+1}) 
-\nabla_{y} \Phi(x^{k_{i}},y^{k_{i}}) \right) \to 0$.

According to \cref{lemma:TMM},  the operator 
$T: \mathcal{H}_{1} \times \mathcal{H}_{2} \to 2^{\mathcal{H}_{1} 
	\times 	\mathcal{H}_{2} }$ defined as 	$(\forall (\bar{x},\bar{y}) \in \mathcal{H}_{1} 
\times \mathcal{H}_{2})$ 
$T (\bar{x},\bar{y})  = \partial_{x}f(\bar{x},\bar{y}) \times \partial_{y}(-f(\bar{x},\bar{y}))$
is maximally monotone. 
Based on results obtained above and 
\cite[Proposition~20.38(ii)]{BauschkCombettes2017}, we derive that 
\[ 
	(0,0) \in  \partial_{x} f (\bar{x},\bar{y}) \times \partial_{y} \left( -f (\bar{x},\bar{y}) 
	\right),
\]
which, via \cite[Fact~2.1]{OuyangGPPA2023}, 
ensures that $ (\bar{x},\bar{y})$ is a saddle-point of $f$.
\end{proof}

 \begin{theorem} \label{theorem:weakconverge} 
 	Let $(x,y)$ be in $\mathcal{H}_{1} \times \mathcal{H}_{2}$, let $(x^{*},y^{*})$ 
 	be a saddle-point of $f$, and let 
 	$(\forall k \in \mathbf{N})$ $a_{k}(x,y)$, $b_{k+1}(x,y)$, and $c_{k}$ be defined in 
 	\cref{lemma:fxk+1yxyk+1:ak}, 
 	\cref{lemma:fxk+1yxyk+1:bk}, and \cref{lemma:fxk+1yxyk+1:ck}, respectively. 
 	Suppose that $(\forall k \in \mathbf{N})$ 
 	$\lambda_{k} =\theta_{k} \leq 1$ and $t_{k}=\frac{t_{0}}{\theta_{1}\cdots \theta_{k}}$ 
 	with
 	$\lambda_{0} =\theta_{0}=1$ and $t_{0} \in \mathbf{R}_{++}$, and
 	that $(\forall k \in \mathbf{N})$ $\tau_{k+1} \geq 
 	\frac{\tau_{k}}{\theta_{k+1} 
 		(1+\mu_{1}\tau_{k})}$ and $\sigma_{k+1} \geq \frac{\sigma_{k}}{\theta_{k+1} 
 		(1+\mu_{2}\sigma_{k})}$.   	Suppose that $(\forall k \in 
 	\mathbf{N})$ $\tau_{k}=\sigma_{k}$ and  
 $\lim_{k \to \infty} \tau_{k} =\lim_{k \to \infty} \sigma_{k} \in \mathbf{R}_{++}$.
 	
 	Suppose that  $\eta_{x}:=1-\sup_{k \in \mathbf{N}}\tau_{k}\left(   
 	L_{xx}\sup_{k \in \mathbf{N}}\lambda_{k}\alpha_{k} 
 	+L_{xy}\sup_{k \in \mathbf{N}}\lambda_{k}\beta_{k} +\frac{L_{xx}}{\inf_{k \in 
 			\mathbf{N}}\alpha_{k}} +\frac{L_{yx}}{\inf_{k \in \mathbf{N}}\gamma_{k}}  
 	\right) \in \mathbf{R}_{++}$
 	and 
 	$\eta_{y}:= 1- \sup_{k \in \mathbf{N}}\sigma_{k} \left(  L_{yx}\sup_{k \in \mathbf{N}} 
 	\theta_{k} \gamma_{k} + 
 	L_{yy}\sup_{k \in \mathbf{N}}\theta_{k}\delta_{k} + \frac{L_{xx}}{\inf_{k \in 
 			\mathbf{N}}\beta_{k}} +\frac{L_{yy}}{\inf_{k \in \mathbf{N}}\delta_{k}}  
 	\right) \in \mathbf{R}_{++}$.

 	Then the iteration sequence 
 	$((x^{k},y^{k}))_{k \in \mathbf{N}}$ converges weakly to a saddle point 
 	$(x^{*},y^{*})$  of $f$.
 \end{theorem} 

\begin{proof}
 Because $(\forall k \in \mathbf{N})$ $\tau_{k}=\sigma_{k}$ and  
$\lim_{k \to \infty} \tau_{k} =\lim_{k \to \infty} \sigma_{k} \in \mathbf{R}_{++}$, by
\cref{prop:proxproxdiff}\cref{prop:proxproxdiff:exists:existence}, we know that 	
$\lim_{k\to \infty} \norm{x^{*}-x^{k}}^{2} + \norm{y^{*}-y^{k}}$ exists.
Combine	this existence result and  
\cref{prop:proxproxdiff}\cref{prop:proxproxdiff:exists:saddleppoint} with 
\cref{fact:weakconvege}
(consider the set $C$ in  \cref{fact:weakconvege} as the set of 
all saddle-points of $f$) to obtain the required result. 
\end{proof}

\subsection{Linear Convergence under Strong Convex Assumptions}
\label{subsection:LinearConvergence}

In this subsection, we assume that both $f_{1}$ and $f_{2}$ are strongly convex, 
that is,  $\mu := \min \{\mu_{1},\mu_{2}\} >0$. 
The following result shows that  the 
sequence $((x^{k},y^{k}))_{k \in \mathbf{N}}$ of iterations converges linearly to a 
saddle-point $(x^{*},y^{*})$ of $f$.
\begin{theorem}\label{theorem:linearconverge}
	Let $(x,y)$ be in $\mathcal{H}_{1} \times \mathcal{H}_{2}$, let $(x^{*},y^{*})$ 
	be a saddle-point of $f$, and let 
	$(\forall k \in \mathbf{N})$ $a_{k}(x,y)$, $b_{k+1}(x,y)$, and $c_{k}$ be defined in 
	\cref{lemma:fxk+1yxyk+1:ak}, 
	\cref{lemma:fxk+1yxyk+1:bk}, and \cref{lemma:fxk+1yxyk+1:ck}, respectively. 
	Suppose  that 
	$(\forall k \in \mathbf{N})$ $\tau_{k}=\sigma_{k} \equiv \sigma \in 
	\mathbf{R}_{++}$, $\alpha_{k} \equiv \alpha \in \mathbf{R}_{++}$, $\beta_{k} \equiv 
	\beta \in \mathbf{R}_{++}$,  
	$\delta_{k} \equiv \delta  \in \mathbf{R}_{++}$, $\gamma_{k} \equiv 
	\gamma \in \mathbf{R}_{++}$,  and $\lambda_{k}=\theta_{k} =\theta = 
	\frac{1}{1+\mu \sigma} \in (0,1)$ with $\mu := \min \{\mu_{1},\mu_{2}\} >0$,
	that $(\forall k \in \mathbf{N})$ 
	$t_{k}=\frac{t_{0}}{ \theta^{k}}$ with $t_{0} \in \mathbf{R}_{++}$,
and  that  $\eta_{x}:= 1- \sigma \left(  
\frac{1}{1+\mu \sigma}	\left( 	L_{xx} \alpha 
	+L_{xy} \beta  \right)
	+\frac{L_{xx}}{\alpha} +\frac{L_{yx}}{\gamma}  
	\right) \in \mathbf{R}_{++}$
	and 
	$\eta_{y}:= 1- \sigma  \left(
	\frac{1}{1+\mu \sigma} \left( L_{yx} \gamma  + 
	L_{yy} \delta  \right)+ \frac{L_{xy}}{ \beta} +\frac{L_{yy}}{\delta}  
	\right) \in \mathbf{R}_{++}$. 

Let $K \in \mathbf{N}\smallsetminus	\{0\}$. 		Then we have the following statements. 
	\begin{enumerate}
		\item \label{prop:linearconverge:xy} 
$
	\scalemath{0.9}{	0\leq 2\theta \sigma \left( f(\hat{x}_{K},y^{*}) 
	-f(x^{*},\hat{y}_{K})\right) + 
		\eta_{x}\norm{x^{*} -x^{K}}^{2}  +\eta_{y}\norm{y^{*} -y^{K}}^{2}  
		\leq  \theta^{K}  \left( 
		\norm{x^{*}-x^{0}}^{2} 
		+ \norm{y^{*}-y^{0}}^{2} \right) }.
$
	
		Consequently, $0 \leq  f(\hat{x}_{K},y^{*}) -f(x^{*},\hat{y}_{K}) \leq 
		\frac{\theta^{K-1}}{2\sigma} 
		\left( 
		\norm{x^{*}-x^{0}}^{2} 
		+ \norm{y^{*}-y^{0}}^{2} \right)$; 
		
		moreover, 
		$\norm{x^{*} -x^{K}}^{2}  + \norm{y^{*} -y^{K}}^{2}  \leq  \theta^{K} 
		\frac{1}{\min\{\eta_{x},\eta_{y}\} }  
		\left( \norm{x^{*}-x^{0}}^{2} 
		+ \norm{y^{*}-y^{0}}^{2} \right)$, 
		that is, the sequence $((x^{k},y^{k}))_{k \in \mathbf{R}}$ of 
		iterations  converges linearly to a saddle-point 
		$(x^{*},y^{*})$ of $f$.

	\item \label{prop:linearconverge:f} 
$
\scalemath{0.9}{	-\frac{\theta^{K-1}}{2\sigma}	\left(  
	\norm{\hat{x}_{K}-x^{0}}^{2} 
	+ \norm{y^{*}-y^{0}}^{2} \right) \leq
	f(\hat{x}_{K},\hat{y}_{K}) -f(x^{*},y^{*})  
	\leq  \frac{\theta^{K-1}}{2\sigma}	\left(
	\norm{x^{*}-x^{0}}^{2} 
	+ \norm{\hat{y}_{K}-y^{0}}^{2} \right)}.
$

Consequently, $(f(\hat{x}_{K},\hat{y}_{K}) )_{k \in \mathbf{N}}$ converges linearly to  
$f(x^{*},y^{*})$.
	\end{enumerate}

\end{theorem}

\begin{proof}
 In view of our assumptions, 
\begin{subequations}\label{eq:prop:linearconverge:t}
	\begin{align}
		&\frac{\sum^{K-1}_{i=0}t_{i}}{t_{0}} = 
		\frac{t_{0}\sum^{K-1}_{i=0}\theta^{-i}}{t_{0}}
		=\frac{1}{\theta^{K-1}} \sum^{K-1}_{i=0}\theta^{i} =\frac{1}{\theta^{K-1}} 
		\frac{1-\theta^{K}}{1-\theta}\geq 
		\frac{1}{\theta^{K-1}};\label{eq:prop:linearconverge:t:t0}\\
		 &\frac{\sum^{K-1}_{i=0}t_{i}}{t_{K}}
		 =\frac{t_{0}}{t_{K}}  \frac{\sum^{K-1}_{i=0}t_{i}}{t_{0}} =\theta^{K} 
		 \frac{\sum^{K-1}_{i=0}t_{i}}{t_{0}} \geq \theta^{K}  \frac{1}{\theta^{K-1}}=  \theta.
	\end{align}
\end{subequations}
Moreover, 
\begin{subequations} \label{eq:prop:linearconverge:geq}
	\begin{align}
 	& \frac{1}{\sigma} -\lambda \left( 
		L_{xx}\alpha +
		L_{xy}\beta	\right) = \frac{\eta_{x}}{\sigma}+\frac{L_{xx}}{\alpha} 
		+\frac{L_{yx}}{\gamma} \geq  \frac{\eta_{x}}{\sigma} >0;  \\
		&\frac{1}{\sigma} -\theta
		\left(L_{yx}\gamma+L_{yy}\delta
		\right) = \frac{\eta_{y}}{\sigma} +\frac{Lxy}{\beta} +\frac{L_{yy}}{\delta} \geq 
		\frac{\eta_{y}}{\sigma} >0.
	\end{align}
\end{subequations}

According to our assumptions, for every $k \in \mathbf{N}$,
\begin{align*}
	&\tau_{k+1}=\sigma \geq \frac{(1+\mu\sigma)\sigma}{1+\mu_{1}\sigma}=
	\frac{\tau_{k}}{\theta_{k+1} 
		(1+\mu_{1}\tau_{k})};\\
	&\sigma_{k+1} =\sigma \geq \frac{(1+\mu\sigma)\sigma}{1+\mu_{2}\sigma}= 
	\frac{\sigma_{k}}{\theta_{k+1} 
		(1+\mu_{2}\sigma_{k})}. 
\end{align*}
	So we are able to apply 
	\cref{lemma:fergodic}\cref{lemma:fergodic:ineq}$\&$\cref{lemma:fergodic:abk:f}.

\cref{prop:linearconverge:xy}: 
Applying
\cref{lemma:fergodic}\cref{lemma:fergodic:ineq}$\&$\cref{lemma:fergodic:abk:f}, 
in the first two inequalities below, we observe that   
\begin{align*}
	&f(\hat{x}_{K},y) -f(x,\hat{y}_{K})\\
	\leq &\frac{1}{\sum^{K-1}_{i=0}t_{i}}  
	\sum^{K-1}_{k=0}t_{k} \left( a_{k}(x,y) -b_{k+1}(x,y) -c_{k}\right)\\
	\leq &\frac{t_{0}}{2\sum^{K-1}_{i=0}t_{i}}   
	\left( \frac{1}{\tau_{0}} 
	\norm{x-x^{0}}^{2} 
	+\frac{1}{\sigma_{0}}\norm{y-y^{0}}^{2} \right)\\
	&-  \frac{t_{K}}{2\sum^{K-1}_{i=0}t_{i}}
	\left(\frac{1}{\tau_{K}} -\lambda_{K} \left( L_{xx}\alpha_{K} + 
	L_{xy}\beta_{K}\right) 
	\right)
	\norm{x -x^{K}}^{2}\\
	&-  \frac{t_{K}}{2\sum^{K-1}_{i=0}t_{i}} 
	\left(  \frac{1}{\sigma_{K}} -\theta_{K} 
	\left(L_{yx}\gamma_{K}+L_{yy}\delta_{K}  
	\right)\right)
	\norm{y -y^{K}}^{2}\\
	\stackrel{\cref{eq:prop:linearconverge:t}}{\leq} & \frac{\theta^{K-1}}{2} \left( 
	\frac{1}{\sigma} 
	\norm{x-x^{0}}^{2} 
	+\frac{1}{\sigma}\norm{y-y^{0}}^{2} \right) -\frac{1}{2 \theta} \left(\frac{1}{\sigma} 
	-\lambda \left( L_{xx}\alpha + 
	L_{xy}\beta\right) 
	\right)
	\norm{x -x^{K}}^{2}\\
	&-\frac{1}{2 \theta} \left(  \frac{1}{\sigma} -\theta
	\left(L_{yx}\gamma+L_{yy}\delta  
	\right)\right)
	\norm{y -y^{K}}^{2}\\
	\stackrel{\cref{eq:prop:linearconverge:geq}}{\leq} & \frac{\theta^{K-1}}{2 \sigma} 
	\left( 
	\norm{x-x^{0}}^{2} 
	+\norm{y-y^{0}}^{2} \right) -\frac{1}{2 \theta} \frac{\eta_{x}}{\sigma}
	\norm{x -x^{K}}^{2} -\frac{1}{2 \theta} \frac{\eta_{y}}{\sigma}
	\norm{y -y^{K}}^{2}.
\end{align*}
After some easy algebra, we get that 
\begin{align} \label{eq:prop:linearconverge}
2\theta  \left( f(\hat{x}_{K},y) -f(x,\hat{y}_{K})\right) 
+ 	\frac{\eta_{x}}{\sigma}
\norm{x -x^{K}}^{2} 
+ \frac{\eta_{y}}{\sigma}
\norm{y -y^{K}}^{2} 
\leq   \frac{\theta^{K}}{\sigma}	\left(  
\norm{x-x^{0}}^{2} 
+ \norm{y-y^{0}}^{2} \right).
\end{align} 

Because $(x^{*},y^{*})$ is a saddle-point of $f$,
via \cref{eq:saddle-point}, we know that 
$ f(\hat{x}_{K},y^{*}) -f(x^{*},\hat{y}_{K}) \geq 0$.
Combine this with \cref{eq:prop:linearconverge} to deduce that 
\begin{align*}
 0\leq &2\theta \sigma \left( f(\hat{x}_{K},y^{*}) -f(x^{*},\hat{y}_{K})\right) + 
	\eta_{x}\norm{x^{*} -x^{K}}^{2}  +\eta_{y}\norm{y^{*} -y^{K}}^{2} \\
	\leq & \theta^{K}  \left( 
	\norm{x^{*}-x^{0}}^{2} 
	+ \norm{y^{*}-y^{0}}^{2} \right), 
\end{align*}
which  yields that
\begin{subequations}
	\begin{align*}
		&0 \leq  f(\hat{x}_{K},y^{*}) -f(x^{*},\hat{y}_{K}) \leq \frac{\theta^{K-1}}{2\sigma} 
		\left( 
		\norm{x^{*}-x^{0}}^{2} 
		+ \norm{y^{*}-y^{0}}^{2} \right);\\
	&\norm{x^{*} -x^{K}}^{2}  + \norm{y^{*} -y^{K}}^{2}  \leq  \theta^{K} 
		\frac{1}{\min\{\eta_{x},\eta_{y}\} }  \left( 
		\norm{x^{*}-x^{0}}^{2} 
		+\norm{y^{*}-y^{0}}^{2} \right).
	\end{align*}
\end{subequations}

\cref{prop:linearconverge:f}: 
According to  \cref{lemma:fxkykx*y*} and 
\cref{lemma:fergodic}\cref{lemma:fergodic:ineq}$\&$\cref{lemma:fergodic:abk:f},  we 
have that 
\begin{align*}
	 f(\hat{x}_{K},\hat{y}_{K}) -f(x^{*},y^{*})  
	 \leq &\frac{1}{\sum^{K-1}_{i=0}t_{i}} \sum^{K-1}_{j=0}t_{j} \left( f(x^{j+1}, 
	 \hat{y}_{K})  
	 -f(x^{*},y^{j+1}) \right)\\
	 \leq & \frac{t_{0}}{2\sum^{K-1}_{i=0}t_{i}}   
	 \left( \frac{1}{\sigma} 
	 \norm{x^{*}-x^{0}}^{2} 
	 +\frac{1}{\sigma}\norm{\hat{y}_{K}-y^{0}}^{2} \right)\\
	 &-  
	 \frac{t_{K}}{2\sum^{K-1}_{i=0}t_{i}}
	 \left(\frac{1}{\sigma} -\lambda \left( L_{xx}\alpha +
	 L_{xy}\beta \right) 
	 \right)
	 \norm{x^{*} -x^{K}}^{2}\\
	 &-  \frac{t_{K}}{2\sum^{K-1}_{i=0}t_{i}} 
	 \left(  \frac{1}{\sigma} -\theta
	 \left(L_{yx}\gamma+L_{yy}\delta
	 \right)\right)
	 \norm{\hat{y}_{K} -y^{K}}^{2}\\
	\stackrel{\cref{eq:prop:linearconverge:geq}}{\leq} &  \frac{t_{0}}{2\sigma 
	\sum^{K-1}_{i=0}t_{i}}   
	\left(  
	\norm{x^{*}-x^{0}}^{2} 
	+ \norm{\hat{y}_{K}-y^{0}}^{2} \right)\\
	\stackrel{\cref{eq:prop:linearconverge:t:t0}}{=} &\frac{\theta^{K-1}}{2\sigma}\left(  
	\norm{x^{*}-x^{0}}^{2} 
	+ \norm{\hat{y}_{K}-y^{0}}^{2} \right).
	\end{align*}
	 Similarly, due to \cref{lemma:fxkykx*y*} and 
	 \cref{lemma:fergodic}\cref{lemma:fergodic:ineq}$\&$\cref{lemma:fergodic:abk:f},  
	 we have that 
\begin{align*}	 
	 f(x^{*},y^{*}) -f(\hat{x}_{K},\hat{y}_{K})
	 \leq &\frac{1}{\sum^{K-1}_{i=0}t_{i}} \sum^{K-1}_{j=0}t_{j} \left( f(x^{j+1},y^{*})  
	 -f(\hat{x}_{K},y^{j+1}) \right) \\
	 \leq & \frac{t_{0}}{2\sum^{K-1}_{i=0}t_{i}}   
	 \left( \frac{1}{\sigma} 
	 \norm{\hat{x}_{K}-x^{0}}^{2} 
	 +\frac{1}{\sigma}\norm{y^{*}-y^{0}}^{2} \right)\\
	 &-  
	 \frac{t_{K}}{2\sum^{K-1}_{i=0}t_{i}}
	 \left(\frac{1}{\sigma} -\lambda \left( L_{xx} \alpha +
	 L_{xy}\beta \right) 
	 \right)
	 \norm{\hat{x}_{K} -x^{K}}^{2}\\
	 &-  \frac{t_{K}}{2\sum^{K-1}_{i=0}t_{i}} 
	 \left(  \frac{1}{\sigma} -\theta
	 \left(L_{yx}\gamma+L_{yy}\delta
	 \right)\right)
	 \norm{y^{*} -y^{K}}^{2}\\
	 	\stackrel{\cref{eq:prop:linearconverge:geq}}{\leq} &
	 	\frac{t_{0}}{2\sigma \sum^{K-1}_{i=0}t_{i}}   
	 	\left( 
	 	\norm{\hat{x}_{K}-x^{0}}^{2} 
	 	+ \norm{y^{*}-y^{0}}^{2} \right)\\
	 		\stackrel{\cref{eq:prop:linearconverge:t:t0}}{=} & \frac{\theta^{K-1}}{2\sigma}
	 		\left(  
	 		\norm{\hat{x}_{K}-x^{0}}^{2} 
	 		+ \norm{y^{*}-y^{0}}^{2} \right).
\end{align*}

Altogether, we have that 
\begin{subequations}\label{eq:prop:linearconverge:f}
	\begin{align}
		&f(\hat{x}_{k},\hat{y}_{k}) -f(x^{*},y^{*}) \leq  \frac{\theta^{k-1}}{2\sigma}	\left(
		\norm{x^{*}-x^{0}}^{2} 
		+ \norm{\hat{y}_{k}-y^{0}}^{2} \right); \\
		&f(x^{*},y^{*}) -f(\hat{x}_{k},\hat{y}_{k})	\leq
		 \frac{\theta^{k-1}}{2\sigma}	\left(  
		\norm{\hat{x}_{k}-x^{0}}^{2} 
		+ \norm{y^{*}-y^{0}}^{2} \right).
	\end{align}
\end{subequations} 
In view of \cref{prop:linearconverge:xy} above, we know that 
$((x^{k},y^{k}))_{k \in \mathbf{N}}$ is bounded, which yields to the boundedness of 
$\left(\left(\hat{x}_{k},\hat{y}_{k}\right)\right)_{k \in \mathbf{N}}$.
This combined with \cref{eq:prop:linearconverge:f} guarantees the desired linear 
convergence result. 
\end{proof}

\begin{remark}\label{remark:linearconvergence}
	Consider assumptions in \cref{theorem:linearconverge}.
	
	Note that  the numbers $\frac{1}{1+\mu \sigma}	\left( 	L_{xx} \alpha 
+L_{xy} \beta  \right)
+\frac{L_{xx}}{\alpha} +\frac{L_{yx}}{\gamma}  $ and 
$\frac{1}{1+\mu \sigma} \left( L_{yx} \gamma  + 
L_{yy} \delta  \right)+ \frac{L_{xy}}{\beta} +\frac{L_{yy}}{\delta}  $ are constants, that 
only $\mu$
	is related to the property of  our functions $f_{1}$ and $f_{2}$, and that $\alpha, 
	\beta, \gamma$, and $\delta$ can be any positive numbers that the user of our 
	algorithm likes. To satisfy 
	requirements of the linear convergence presented in  \cref{theorem:linearconverge},
	we can always set the involved parameters $(\forall k \in \mathbf{N})$ 
	$\tau_{k}=\sigma_{k} \equiv \sigma \in 
	\mathbf{R}_{++}$ in the iteration scheme \cref{eq:algorithmproxif1f2} small enough.
	 
	Therefore,  \cref{assumption} are basically 
	all requirements for our convex-concave saddle-point problems and our algorithm 
	\cref{eq:algorithmproxif1f2} and linear convergence result 
	\cref{theorem:linearconverge} are pretty practical.
	
	Clearly, we have similar conclusions for our convergence results 
	\cref{theorem:fconverge,theorem:weakconverge} in the last subsection.
\end{remark}

  \section*{Acknowledgments}
Hui Ouyang thanks Professor Boyd Stephen for his insight and expertise comments 
 on the topic of saddle-point problems
and all unselfish support. 
Hui Ouyang also thanks Professor Ryu Ernest for some useful conversations. 
Hui Ouyang acknowledges 
the Natural Sciences and Engineering Research Council of Canada (NSERC), 
[funding reference number PDF – 567644 – 2022]. 

\section*{Data Availability Statements}
We do not analyse or generate any datasets,
 because our work proceeds within a theoretical and mathematical approach.

 \addcontentsline{toc}{section}{References}
\bibliographystyle{abbrv}
\bibliography{ccspp}

\begin{thebibliography}{10}

\bibitem{BauschkCombettes2017}
H.~H. Bauschke and P.~L. Combettes.
\newblock {\em Convex analysis and monotone operator theory in {H}ilbert
  spaces}.
\newblock CMS Books in Mathematics/Ouvrages de Math\'{e}matiques de la SMC.
  Springer, Cham, second edition, 2017.

\bibitem{BonettiniRuggiero2012}
S.~Bonettini and V.~Ruggiero.
\newblock On the convergence of primal-dual hybrid gradient algorithms for
  total variation image restoration.
\newblock {\em J. Math. Imaging Vision}, 44(3):236--253, 2012.

\bibitem{BotCsetnekSedlmayer2022accelerated}
R.~I. Bo{\c{t}}, E.~R. Csetnek, and M.~Sedlmayer.
\newblock An accelerated minimax algorithm for convex-concave saddle point
  problems with nonsmooth coupling function.
\newblock {\em Computational Optimization and Applications}, pages 1--42, 2022.

\bibitem{ChambollePock2011}
A.~Chambolle and T.~Pock.
\newblock A first-order primal-dual algorithm for convex problems with
  applications to imaging.
\newblock {\em J. Math. Imaging Vision}, 40(1):120--145, 2011.

\bibitem{HamedaniAybat2021}
E.~Y. Hamedani and N.~S. Aybat.
\newblock A primal-dual algorithm with line search for general convex-concave
  saddle point problems.
\newblock {\em SIAM J. Optim.}, 31(2):1299--1329, 2021.

\bibitem{HeYouYuan2014}
B.~He, Y.~You, and X.~Yuan.
\newblock On the convergence of primal-dual hybrid gradient algorithm.
\newblock {\em SIAM J. Imaging Sci.}, 7(4):2526--2537, 2014.

\bibitem{HiriartJeanLemarechal1993}
J.-B. Hiriart-Urruty and C.~Lemar\'{e}chal.
\newblock {\em Convex analysis and minimization algorithms. {I}}, volume 305 of
  {\em Grundlehren der mathematischen Wissenschaften [Fundamental Principles of
  Mathematical Sciences]}.
\newblock Springer-Verlag, Berlin, 1993.

\bibitem{OuyangGPPA2023}
H.~Ouyang.
\newblock Weak and strong convergence of generalized proximal point algorithms
  with relaxed parameters.
\newblock {\em J. Global Optim.}, 85(4):969--1002, 2023.

\bibitem{RockafellarMinimax1964}
R.~T. Rockafellar.
\newblock Minimax theorems and conjugate saddle-functions.
\newblock {\em Math. Scand.}, 14:151--173, 1964.

\bibitem{RockafellarMonotone1970}
R.~T. Rockafellar.
\newblock Monotone operators associated with saddle-functions and minimax
  problems.
\newblock In {\em Nonlinear {F}unctional {A}nalysis ({P}roc. {S}ympos. {P}ure
  {M}ath., {V}ol. {XVIII}, {P}art 1, {C}hicago, {I}ll., 1968)}, pages 241--250.
  Amer. Math. Soc., Providence, R.I., 1970.

\bibitem{RockafellarSaddlePoints1971}
R.~T. Rockafellar.
\newblock Saddle-points and convex analysis.
\newblock In {\em Differential {G}ames and {R}elated {T}opics ({P}roc.
  {I}nternat. {S}ummer {S}chool, {V}arenna, 1970)}, pages 109--127.
  North-Holland, Amsterdam, 1971.

\bibitem{ZhuChan2008efficient}
M.~Zhu and T.~Chan.
\newblock An efficient primal-dual hybrid gradient algorithm for total
  variation image restoration.
\newblock {\em Ucla Cam Report}, 34:8--34, 2008.

\end{thebibliography}
 
\end{document}